\documentclass[11pt,reqno]{amsart}    

\usepackage{amsmath,amssymb}
\usepackage{graphicx,xcolor,mathrsfs}
\usepackage[font=scriptsize,labelfont=bf]{caption}
\usepackage{cite}

\renewcommand{\geq}{\geqslant}
\renewcommand{\leq}{\leqslant}

\newcommand{\z}{\zeta}
\newcommand{\RM}{\mathbb{R}}
\newcommand{\ZM}{\mathbb{Z}}
\newcommand{\NM}{\mathbb{N}}
\newcommand{\CM}{\mathbb{C}}
\newcommand{\zygmund}{\mathcal{C}}
\newcommand{\SM}{\mathbb{S}}
\newcommand{\Lop}{\mathcal{L}}

\newcommand{\diff}{\,\mathrm{d}}

\DeclareMathOperator{\ran}{ran}
\DeclareMathOperator{\kernel}{ker}
\DeclareMathOperator{\spn}{span}

\newtheorem{theorem}{Theorem}[section]
\newtheorem{lemma}[theorem]{Lemma}
\newtheorem{proposition}[theorem]{Proposition}
\newtheorem{remark}[theorem]{Remark}

\newtheorem{corollary}[theorem]{Corollary}

\newtheorem*{main-theorem}{Main Theorem}
\newtheorem*{remark*}{Remark}
\numberwithin{equation}{section}

\title[Bifurcations in the capillary-gravity Whitham equation]{On the bifurcation diagram of the capillary-gravity Whitham equation}

\author[Ehrnstr\"om]{Mats Ehrnstr\"om}
\address{Department of Mathematical Sciences, Norwegian University of Science and Technology, 7491 Trondheim, Norway}
\thanks{All authors acknowledge the support by grant nos. 231668 and 250070 from the Research Council of Norway. M.J. was supported by the National Science Foundation under grant DMS-1614785.}
\email{mats.ehrnstrom@ntnu.no}

\author[Johnson]{Mathew~A.~Johnson}
\address{Department of Mathematics, University of Kansas, Lawrence, KS 66045 USA} 
\email{matjohn@ku.edu}

\author[M{\ae}hlen]{Ola I.H. Maehlen}
\address{Department of Mathematical Sciences, Norwegian University of Science and Technology, 7491 Trondheim, Norway}
\email{ola.mahlen@ntnu.no}

\author[Remonato]{Filippo Remonato}
\address{Department of Mathematical Sciences, Norwegian University of Science and Technology, 7491 Trondheim, Norway \emph{and} Department of Mathematics, University of Pavia, 27100 Pavia, Italy}
\email{filippo.remonato@ntnu.no}

\date{\today}


\usepackage[colorlinks=true]{hyperref}
\definecolor{darkred}{rgb}{0.55,0,0}
\hypersetup{urlcolor=blue, citecolor=darkred, linkcolor=blue}

\begin{document}

\maketitle

\begin{abstract}
We study the bifurcation of periodic travelling waves of the capillary-gravity Whitham equation. This is a nonlinear pseudo-differential equation that combines the canonical shallow water nonlinearity with the exact (unidirectional) dispersion for finite-depth capillary-gravity waves. Starting from the line of zero solutions, we give a complete description of all small periodic solutions, unimodal as well bimodal, using simple and double bifurcation via Lyapunov--Schmidt reductions. Included in this study is the resonant case when one wavenumber divides another. Some bifurcation formulas are studied, enabling us, in almost all cases, to continue the unimodal bifurcation curves into global curves. By characterizing the range of the surface tension parameter for which the integral kernel corresponding to the linear dispersion operator is completely monotone (and therefore positive and convex; the threshold value for this to happen turns out to be  \(T = \frac{4}{\pi^2}\), not the critical Bond number \(\frac{1}{3}\)), we are able to say something about the nodal properties of solutions, even in the presence of surface tension. Finally, we present a few general results for the equation and discuss, in detail, the complete bifurcation diagram as far as it is known from analytical and numerical evidence. Interestingly, we find, analytically, secondary bifurcation curves connecting different branches of solutions; and, numerically, that all supercritical waves preserve their basic nodal structure, converging asymptotically  in \(L^2(\SM)\) (but not in \(L^\infty\)) towards one of the two constant solution curves.
\end{abstract}

\section{Introduction}
We consider periodic travelling wave solutions of the capillary-gravity Whitham equation
\begin{equation}\label{e:w1}
u_t+M_Tu_x+2uu_x=0
\end{equation}
where $M_T$ is a Fourier multiplier operator defined via its symbol \(m_T\) as
\begin{equation}\label{e:M}
\widehat{M_Tf}(\xi)= m_T(\xi)\widehat{f}(\xi) = \left(\frac{(1+T\xi^2)\tanh(\xi)}{\xi}\right)^{\frac{1}{2}}~\widehat{f}(\xi),
\end{equation}
and the coefficient $T>0$ denotes the strength of the surface tension. The symbol \(m_T\) arises as the linear dispersion relation for capillary-gravity water waves over a finite depth described by the Euler equations \cite{MR3060183}. In the purely gravitational case, that is, when $T=0$, the use of this symbol was proposed by Whitham as a way to generalise the KdV equation and remedy its strong dispersion \cite{Whitham1967vma}.
Bifurcation in the gravitational setting has been investigated in \cite{Ehrnstrom2013gbf, Ehrnstrom2009twf,Ehrnstrom2016owc}.
We are here interested in completely characterising the local theory for travelling wave solutions of \eqref{e:w1}, and understanding their global extensions.

The overarching technique follows an approach similar to that used for the gravity Whitham equation in \cite{Ehrnstrom2013gbf} and the Euler equations in \cite{eew10}, where a Lyapunov--Schmidt reduction is used  to prove the existence of wave solutions through the application of the implicit function theorem.
Here, however, the symbol of the linear dispersion has a different large-frequency behaviour: whereas it is $\sim|\xi|^{-1/2}$ in the gravity case, it changes to $\sim\lvert\xi\rvert^{1/2}$ in the presence of surface tension. Inspired by recent work on large waves for very weakly dispersive equations, we tackle the equation by inverting the linear operator, see \eqref{eq:whitham}, presenting us with a smoothing operator with good properties but that now acts nonlocally on a nonlinear term. Apart from the results presented in this paper, we see this as a first step toward handling large-amplitude theory for equations with mixed nonlocal and nonlinear terms. A study in that direction, but with a different order and global structure of the solutions, has been carried out in \cite{arXiv:1808.08057}.

The organisation of the paper correspond to the development of our theory:

We start, in Section~\ref{sec:kernel}, with a study of the \emph{inverse} of the Fourier multiplier operator \(M\) in \eqref{e:M}. This is a smoothing operator of order \(-\frac{1}{2}\) on any Fourier-based scale of functions spaces (such as the Sobolev and Zygmund spaces), that is realised as a convolution operator with a surface tension-dependent integral kernel \(K_T\).  We characterise the kernel \(K_T\) in Theorem~\ref{thm:regularity and decay}, expressing it as a sum of three terms that are, optimally, in the regularity classes \(\zygmund^{-\frac{1}{2}}\), \(\zygmund^{-\frac{3}{2}}\) and \(C^{\omega}\), respectively, where \(\zygmund^s\) is the scale of Zygmund spaces, and \(C^\omega\) is the class of real-analytic functions. This is different from the regular Whitham symbol which, although of the same order, has only two canonical when decomposed in the same manner \cite{Ehrnstrom2016owc}. As in \cite{Ehrnstrom2016owc} we apply complex analysis techniques and the theory of Stieltjes functions to determine further properties of the convolution kernel, in particular the signs of its derivatives to infinite order. When the surface tension is big enough, \(T \geq \frac{4}{\pi^2}\), we are able in Theorem~\ref{thm:complete monotonicity} to show that the kernel is \emph{completely monotone}, a delicate structural property shared by the kernel for the linear dispersion in the pure gravity case (not its inverse). Moreover, we can show that neither complete monotonicity or monotonicity on a half-line is preserved if \(0 < T <  \frac{4}{\pi^2}\), showing in effect that the critical Bond number \(\frac{1}{3}\) separating weak from strong surface tension is \emph{not} the break-off value for the positivity of the kernel (or its stronger properties).
 Finally, we give in Section~\ref{sec:kernel} the decay rate of \(K_T\) and its compactness properties in suitable spaces. 

In Section~\ref{sec:onedim} we perform the one-dimensional bifurcation of periodic waves from simple eigenvalues along the line of zero solutions.  After an initial discussion of the eigenvalues of the linearised operator, and a scaling to reduce the problem to a fixed period, we use Lyapunov--Schmidt reduction to prove the existence of small-amplitude solutions in a vicinity of the simple eigenvalues (expressed using the wavespeed) in Theorem~\ref{thm:1Dlocal}. The constructed waves are all unimodal and bell-shaped in a minimal period. They arise for both strong and weak surface tension; for strong surface tension they are the only type of waves in a \(\zygmund^s(\SM)\)-vicinity of the line of zero solutions, \(s > 0\). Although one could have carried out the simple bifurcation using the Crandall--Rabinwitz theorem \cite{MR2004250}, we choose to prove Theorem~\ref{thm:1Dlocal} using a Lyapunov--Schmidt reduction as a preparation for the two-dimensional case (which would otherwise be harder to understand). Under a simple condition that relates the wavenumber to the surface tension and period, we prove the continuation of the local solution curves to global ones in Theorem~\ref{thm:1Dglobal}. This condition may be related to sub- and supercritical bifurcation, and we see in Remark~\ref{rem:cdotdot} that both cases can appear.  The modulational stability of these waves in the small-amplitude case has been studied in \cite{Hur2015mii}

A challenge and interesting feature of the capillary-gravity case is that weak surface tension allows for a non-monotone dispersion relation (see Figure~ \ref{F:symbol}) and double eigenvalues of the corresponding linearised operator (in spaces of even functions). We handle this case in Section~\ref{sec:twodim}. To analytically capture the larger dimension of the space of solutions nearby the trivial ones, one requires an additional free parameter in addition to the wavespeed, used in the one-dimensional bifurcation. In line with \cite{Ehrnstrom2015tsw} we choose to use the period as this extra parameter, while holding the surface tension fixed. The result, presented in Theorem~\ref{thm:2DLocal}, depends on the resonances between the two frequencies appearing in the nullspace: if one of the wavenumbers is a multiple of the other, one obtains a slit disk of solutions, excluding bifurcation straight in the direction of the higher wavenumber; if not, one obtains a full open disk of solutions, see Figure~\ref{fig:solutionDisk}. These results are in line with similar ones in \cite{eew10,JonesToland85,Reeder1981owr}, and include --- when projecting the full disk onto a fixed period --- a curve of bimodal rippled waves connecting waves of different wavenumbers (secondary bifurcation). The existence of these interconnecting branches of waves have been corroborated numerically, showing persistence with respect to perturbations in the surface tension parameter \cite{Remonato2017nbf}. The nonexistence of the pure higher mode in the resonant case of Theorem~\ref{thm:2DLocal}~(ii) has also been confirmed numerically in the same paper. More generally, Wilton ripples, as these kinds of waves are sometimes called, have earlier been found to exist for the Euler equations with surface tension  \cite{Reeder1981owr,JonesToland85}, and their spectral stability has been numerically investigated in \cite{Trichtchenko2016tio}. They also exist in the presence of vorticity \cite{Martin2013eow}, even without capillarity \cite{eew10,Ehrnstrom2015tsw}. In that case, one may even construct arbitrary large kernels \cite{07022009,Aasen2018tgw}, and corresponding multi-dimensional solution sets \cite{Kozlov2017nms}.

Our motivation for this investigation has arisen from two different directions: one is the study of the (very) weakly dispersive equations with nonlocal nonlinearities, and especially their large-amplitude theories; the other is the mathematically qualitative analogues between the full water-wave problem and the family of fully dispersive Whitham-type equations. While numerical bifurcation of steady water waves with surface tension have been earlier carried out \cite{Aston1991lag}, and display striking resemblances to our case, it is not known how to control the waves along the bifurcation curves when surface tension is present, and our results show that, at least for weak surface tension, the looping alternative in Theorem~\ref{thm:1Dglobal} is possible. Our initial hope was that, using methods as in \cite{Ehrnstrom2016owc,EJC18}, one would be able to say something for larger waves. In Section~\ref{sec:discussion} we turn to this question, as well as discussing the general picture of bifurcation in the capillary-gravity Whitham equation.
While we are indeed able to say something, preserving the nodal properties to \(\mathcal{O}(1)\)-height of the solutions in Proposition~\ref{prop:leaving_c/2}, the final evolution of solution curves is very challenging to handle analytically. While both our preliminary calculations and numerical simulations for this paper indicates that one can follow curves of supercritical bell-shaped solutions all the way to \(c \to \infty\), and that they converge, asymptotically in \(L^2(\SM)\), towards the curve of constant solutions \(u = c -1\), they do not converge in \(L^\infty\), and the analysis is complicated by that the equation lies exactly at the Sobolev-critical balance \(s = \frac{1}{2}\), \(p = 2\) and \(n = 1\). We discuss  both our findings and conjectures in detail in Section~\ref{sec:discussion}. For a quick overview, we refer to Figures~\ref{fig:strong} and~\ref{fig:weak}.

Finally, we give in Appendix~\ref{app:A} some bifurcation formulas.

\section{Properties of the convolution kernel $K_T$}
\label{sec:kernel}

Traveling-wave solutions of the form $u(x-ct)$ satisfy the (profile) equation
\begin{equation}\label{e:profile1}
-cu+M_Tu+u^2=0,
\end{equation}
where we have integrated once and used Galilean invariance to set the constant of integration to zero.
Since $m_T$ is strictly positive on $\RM$, the operator $M_T$ is invertible (for example in any Fourier-based space) with inverse  $L_T$
defined via
\begin{equation}
	\label{eq:L}
	\widehat{L_Tf}(\xi)=l_T(\xi)\widehat{f}(\xi), \qquad l_T(\xi)=(m_T(\xi))^{-1}.
\end{equation}
In particular, the capillary-gravity Whitham equation \eqref{e:profile1} can be rewritten in the ``smoothing" form
\begin{equation}
	\label{eq:whitham}
	u-cL_T(u)+L_T(u^2)=0,
\end{equation}
where $L_T = K_T\ast$ and $K_T$ is the convolution kernel corresponding to the symbol \(l_T\). Note that the form \eqref{eq:whitham} is resemblant of the Whitham equation itself, but with a nonlocal nonlinearity.  \emph{By a solution of \eqref{e:profile1} (respectively \eqref{eq:whitham}), we shall mean a real-valued, continuous and bounded function $u$ that satisfies \eqref{e:profile1} (respectively \eqref{eq:whitham}) everywhere}.

In the rest of this work we shall make heavy use of the properties of the convolution kernel $K_T$ and its symbol. Our choice of Fourier transform is 
\[
\hat{f}(\xi)=\int_{\RM}f(x)e^{-ix\xi}\,\diff x.
\]

To start, note that $K_T=\mathcal{F}^{-1}l_T$  is smooth away from the origin with
\begin{equation}
	\label{eq:symbol_limit_at_zero}
	\int_\RM K_T(x) \diff x = \lim_{\xi\to 0}l_T(\xi)=1
\end{equation}
and
\begin{equation*}
	\lim_{x\to 0}K_T(x) = \frac{1}{2\pi}\int_{\RM} l_T(\xi) \diff \xi = +\infty.
\end{equation*}
Moreover, since $l_T$ is analytic, $K_T$ has rapid decay at $\pm\infty$, whence $K_T\in L^1(\RM)$ provided that the blow-up at $x=0$ is not too fast.
In what follows, we will show that the singularity at the origin is of order $\lvert x \rvert^{-\frac{1}{2}}$ (there is a lower-order singularity appearing as well), and that the convolution kernel is completely monotone for strong enough surface tension.

\subsection{Montonicity and complete monotonicity}
A function $g : (0, \infty) \rightarrow [0,\infty)$ is called \emph{completely monotone} if $g$ is infinitely differentiable with
	\begin{equation*}
		(-1)^n g^{(n)}(\lambda) \geqslant 0
	\end{equation*}
	for $n = 0,1,2,\ldots$ and all $\lambda > 0$. If it can furthermore be written in the form
	\begin{equation*}
		g(\lambda) = \frac{a}{\lambda} + b + \int_{(0, \infty)} \frac{1}{\lambda + t} \diff \sigma(t)
	\end{equation*}
	for some constants $a,b>0$, with $\sigma$ a Borel measure satisfying $\int_{(0, \infty)} \frac{1}{1+t} \diff \sigma(t) < \infty$, then it is called \emph{Stieltjes}.   Our interest in such functions is motivated by the following two results, taken from \cite{Ehrnstrom2016owc} and \cite{Schilling2012bf}.

\begin{lemma}{\rm \cite{Ehrnstrom2016owc}}
	\label{thm:complMonCrit}
	Let $f : \RM \rightarrow \RM$ and $g : (0, \infty) \rightarrow \RM$ be two functions satisfying $f(\xi) = g(\xi^2)$ for $\xi \neq 0$. Then $f$ is the Fourier transform of an even, integrable, and completely monotone function if and only if $g$ is Stieltjes with  $\lim_{\lambda \searrow 0} g(\lambda) < \infty$ and $\lim_{\lambda \rightarrow \infty} g(\lambda) = 0$.
\end{lemma}

\begin{lemma}{\rm \cite{Schilling2012bf}} \label{lemma:stieltjes}
	Let $g$ be a positive function on $(0, \infty)$. Then $g$ is Stieltjes if and only if $\lim_{\lambda \searrow 0} g(\lambda)$ exists in $[0, \infty]$ and $g$ extends analytically to $\CM \setminus (-\infty, 0]$ such that $\mathrm{Im}(z)\cdot \mathrm{Im}(g(z)) \leqslant 0$.
\end{lemma}
With $f(\xi) = l_T(\xi)$ and $g(\xi) = l_T(\sqrt{\xi})$ we want to employ the two above results to conclude that $K_T=\mathcal{F}^{-1}(l_T(\xi))$ is completely monotone
for $T$ sufficiently large. Since $l_T$ has a unit limit at the origin and a vanshing limit at infinity, it only remains to prove that $l_T \circ \sqrt{\cdot}$ is Stieltjes. To this end, define 
\begin{equation}\label{eq:rho}
	\varrho_T(\z) = \frac{\z}{(1+T\z^2) \tanh(\z)},
\end{equation}
with \(\zeta\) a complex number. We are interested in $l_T = \sqrt{\varrho_T}$, $\sqrt{\cdot}$ denoting the principal branch of the square root, and thus want to determine the pre-image of \((-\infty,0)\) together with the singularities of $\varrho_T$. Let furthermore
	\begin{align*}
		Z_c &= \left\{ \pi (k - \textstyle\frac{1}{2}) \colon k \in \ZM \right\}, \\
		Z_s &= \left\{ \pi k \colon k \in \ZM\setminus\{0\} \right\}, \\
		Z_T &= \left\{ {\textstyle{-\frac{1}{\sqrt{T}}, \frac{1}{\sqrt{T}}}} \right\},
	\end{align*}
denote the set of zeros of $\cos(\zeta)$, $\frac{\sin(\zeta)}{\zeta}$, and $1-T\zeta^2$, respectively. 
Finally, recall that the \emph{symmetric difference} between two sets \(A\) and \(B\) is the set  \(A \bigtriangleup B\) of elements either in \(A\) and not \(B\), or contrariwise\footnote{That is,
$(A \bigtriangleup B)=\left(A\cap B^c\right)\cup\left(B\cap A^c\right)$.}

\begin{lemma}\label{lemma:singularities}
Let $\zeta = \xi + i \eta$. Then $\varrho_T(\zeta)$ takes a zero or infinite value exactly if $\xi=0$ and $\eta \in Z_s \cup \left( Z_c \bigtriangleup Z_T \right)$.
Further, $\varrho_T(\zeta)$ is negative exactly when the following three conditions hold:  $\xi=0$,  $\eta \notin Z_s \cup \left( Z_c \bigtriangleup Z_T \right)$, and 
the intersection $(0, \lvert \eta \rvert ) \cap \big( \left( Z_c \cup Z_s\right) \bigtriangleup Z_T \big)$ contains an odd number of elements.
\end{lemma}

\begin{proof}
	By the infinite product formulas for $\sinh\z$ and $\cosh\z$ we obtain
	\begin{equation}\label{eq:infiniteProductFormulaForRhoT}
		\varrho_T(\z)  =
		\frac{1}{1+T\z^2}\prod_{n=1}^\infty \frac{  1 + \frac{\z^2}{\pi^2 (n - \frac{1}{2})^2}  }{ 1 + \frac{\z^2}{\pi^2 n^2}  }.
	\end{equation}
	The first part of the lemma now follows immediately, where the symmetric difference accounts for removable singularities should the term $(1+T\z^2)$ coincide with a term of the form $1 + \frac{\z^2}{\pi^2 (n - \frac{1}{2})^2}$. For the second part we start by showing that $\varrho_T$ is never negative away from the imaginary axis. As $\varrho_T$ is symmetric about zero, we restrict our attention to $\xi>0$. We have 
	\begin{align*}
		\mathrm{Re}\Big[\cosh(\z)\overline{\sinh(\z)}\Big] &= \frac{1}{2}\sinh(2\xi)>0, \\
		\mathrm{Re}\Big[\z\,\overline{(1+T\z^2)}\Big] &= \xi + \xi T(\xi^2+\eta^2)>0,
	\end{align*}
and consequently $|\arg(\frac{\z}{1+T\z^2})|,|\arg(\frac{1}{\tanh(\z)})|<\frac{\pi}{2}$. This in turn implies that $|\arg(\varrho_T(\zeta))|<\pi$, and so $\varrho_T(\zeta)$ cannot be negative. Restricting our attention to the imaginary axis $(\z=i\eta)$ and away from zeroes and singularities, it is clear from \eqref{eq:infiniteProductFormulaForRhoT} that $\varrho_T(i\eta)$ is real valued and satisfies 
	\begin{align*}
	\text{sgn}(\varrho_T(i\eta)) = \text{sgn}(1-T\eta^2)\prod_{n=1}^\infty \text{sgn}\Big(1 - \frac{\eta^2}{\pi^2 (n - \frac{1}{2})^2}\Big)\text{sgn}\Big( 1 - \frac{\eta^2}{\pi^2 n^2}\Big).
	\end{align*}
	As $\varrho_T(i\eta)$ is positive for $\eta=0$, it is negative exactly when an odd number of factors in the expression above has swapped sign. This is equivalent to the last part of the lemma.
\end{proof}

According to Lemma~\ref{lemma:singularities} the real-valued function $l_T$ can be extended analytically as $l_T = \sqrt{\varrho_T}$ outside of 
the zeroes and singularities of $\varrho_T$ along the imaginary axis.
By noting that the function $\frac{\cosh(\zeta)}{1+T\zeta^2}$ has a removable singularity at $\zeta=i\pi/2$ when $T=4/\pi^2$, as well as the fact
that, with $\zeta=\xi+i\eta$, $\sqrt{\CM \setminus (-\infty, 0]} = \CM_{\xi>0}$, we can record the following result. 

\begin{corollary}\label{cor:singularityDistance}
The symbol $l_T$ extends analytically onto the strip $\RM \times i(-\delta^*,\delta^*)$, where
	\begin{equation*}
		\delta^* =
		\begin{cases}
			\min \{ \frac{1}{\sqrt{T}}, \frac{\pi}{2} \}, \qquad & T \neq 4/\pi^2, \\
			\pi \qquad & T = 4/\pi^2.
		\end{cases}
	\end{equation*}
Hence, the  function $\z \mapsto \sqrt{\varrho_T(\sqrt{\z})}$ is the unique analytic extension of $l_T \circ \sqrt{\cdot}$ to $\CM \setminus (-\infty, 0]$.
\end{corollary}

We are now ready to prove Theorem \ref{thm:complete monotonicity}, where we determine a critical value $T_*=\frac{4}{\pi^2}$ of the surface tension 
\(T\), for which $K_T$ is completely monotone whenever $T\geq T_*$.
Note that $T_*$ \emph{does not} correspond to the, likewise critical, Bond number \(T = \frac{1}{3}\) that separates strong 
from weak surface tension; in fact, $T_*>\frac{1}{3}$.  Further, this result is sharp since, as we will see, $K_T$ is not monotone for $T\in(0,T_*)$.
To establish this, we make use of the class of so-called positive definite functions.  A function
$f:\RM\to\CM$ is said to be \emph{positive definite} if for every $n\in\NM$ and \(\boldsymbol\xi\in\RM^n\) the $n\times n$ matrix $[f(\xi_i-\xi_j)]_{i,j=1}^n$ is positive semi-definite.  
We point out the following standard results \cite{Bhatia2007pdf}.

\begin{lemma}\label{L:Positive_Def_Facts}
The following statements are true.
\begin{itemize}
\item[(i)] [Bochner's Theorem] Any positive definite function is the Fourier transform of a non-negative, finite Borel measure.
\item[(ii)] [Shur's Theorem] A countable product of positive definite functions is positive definite.
\item[(iii)] If $f:\RM\to\CM$ is positive definite, then the global maximum of $f$ occurs at $x=0$.
\item[(iv)] The function $f(x)=\frac{1+ax^2}{1+bx^2}$ is positive definite if and only if $b\geq a\geq 0$.
\end{itemize}
\end{lemma}

With the above preliminaries, we now state the main result for this section.

\begin{theorem}\label{thm:complete monotonicity}
For $T\geq\frac{4}{\pi^2}$, the kernel $K_T$ is completely monotone on $(0,\infty)$.  Further, for $0<T<\frac{4}{\pi^2}$ the kernel $K_T$ is not monotone on $(0,\infty)$.
\end{theorem}

\begin{proof}
We first prove that \(K_T\) is completely monotone for \(T\geq\frac{4}{\pi^2}\).
By Lemma \ref{thm:complMonCrit} and Lemma \ref{lemma:stieltjes} and the discussion thereafter, we conclude that \(K_T\) is completely monotone exactly if $\mathrm{Im}(\z)\cdot \mathrm{Im}\sqrt{\varrho_T(\sqrt{\z})} \leqslant 0$ for $\z\in\CM \setminus (-\infty, 0]$. Moreover, this last property is satisfied exactly when it is satisfied for $\varrho_T\circ\sqrt{\cdot}$. 
Moving the first factor of \(\cosh\zeta\) out of the infinite product in \eqref{eq:infiniteProductFormulaForRhoT}, we obtain
\begin{align}\label{eq:rhoRewritten}
    \varrho_T(\xi)=\frac{1+\frac{4}{\pi^2}\xi^2}{1+T\xi^2}\prod_{n=1}^\infty \frac{  1 + \frac{\xi^2}{\pi^2 (n + \frac{1}{2})^2}  }{ 1 + \frac{\xi^2}{\pi^2 n^2}  }.
\end{align}
Substituting \(\xi\mapsto\sqrt{\zeta}\) in \eqref{eq:rhoRewritten}, and taking the complex argument of both sides, we obtain
\begin{equation}\label{eq:argumentOfRhoT}
\begin{split}
	\arg\big(\varrho_T\big(\sqrt{\z}\big)\big) =& \Big[\arg\Big(1+\frac{4}{\pi^2}\z\Big)-\arg(1+T\z)\Big]\\
	&+\sum_{n=1}^\infty \Big[\arg\Big(1 + \frac{\z}{\pi^2 (n + \frac{1}{2})^2}\Big)-\arg\Big( 1 + \frac{\z}{\pi^2 n^2}\Big)\Big].
	\end{split}
\end{equation}
This equation is valid whenever the right hand side takes values in $(-\pi,\pi)$, which in turn is always true in $\zeta\in\CM\setminus(-\infty,0]$ as it is continuous in $\zeta$, zero for $\z>0$ and prevented from taking the values \(\pm\pi\) as $\varrho_T(\sqrt{\z})$ is never negative (Lemma \ref{lemma:singularities}). When Im$(\z) > 0$, it is easy to see that $\alpha\mapsto\arg(1+\alpha\z)$ is strictly increasing for $\alpha>0$, and so each square bracket in \eqref{eq:argumentOfRhoT} is negative (the first non-positive), further implying $\mathrm{Im}(\z)\cdot \mathrm{Im}\sqrt{\varrho_T(\sqrt{\z})} < 0$. After a similar argument for Im$(\z)<0$, we obtain the desired conclusion.  

We now prove that $K_T$ is not a monotone function on $(0,\infty)$ for $0<T<\frac{4}{\pi^2}$.  
Since Theorem \ref{thm:regularity and decay} guarantees that \(K_T\) is  positive near zero and decays to zero at infinity, 
the existence of a point \(K_T(x_0)<0\) would rule out monotonicity of \(K_T\). To this end, we note by Bochner's theorem in Lemma \ref{L:Positive_Def_Facts}(i)
that $K_T$ is non-negative if and only if its Fourier transform $l_T$ is a positive definite function; we now prove this is false when \(0<T<\frac{4}{\pi^2}\). 
Note first that for $0<T<\frac{1}{3}$, this follows immediately from Lemma \ref{L:Positive_Def_Facts}(iii) as $l_T$ does not have a global maximum at $\xi=0$ (see Figure \ref{F:symbol}).
Suppose instead that \(\frac{1}{3}\leq T<\frac{4}{\pi^2}\). If $l_T$ is positive definite, then Lemma \ref{L:Positive_Def_Facts}(ii) implies the same would be true for its square 
\(\xi\mapsto\varrho_T(\xi)\). To this end, we write \eqref{eq:rhoRewritten} as
\[
\varrho_T(\xi)=\frac{1+\frac{4}{\pi^2}\xi^2}{1+T\xi^2}~\varphi(\xi),
\]
which, after introducing the positive constants 
$\alpha=4/(T\pi^2)$ and $\beta=\alpha-1$, 
can be further rewritten as 
\begin{align*}
    \varrho_T(\xi)=\Big(\alpha-\frac{\beta}{1+T\xi^2}\Big)\varphi(\xi)=\colon\alpha\varphi(\xi)-\beta\psi(\xi).
\end{align*}
By Lemma \ref{L:Positive_Def_Facts}, both \(\varphi\) and \(\psi\) are positive definite as they are (countable) products of positive definite functions, and thus \(\hat{\varphi},\hat{\psi}\geq0\) by Bochner's Theorem.  
Note that \(\varphi\) has a complex analytic extension to the strip \(\RM\times i(-\pi,\pi)\), 
while \(\psi\) can not be extended to a larger strip than \(\RM\times i(\frac{-1}{\sqrt{T}},\frac{1}{\sqrt{T}})\), and so by the Paley-Wiener theorem, we have 
\[
0<\int_{\RM}\widehat\varphi(x)e^{\gamma x}\,\diff x<\infty\quad{\rm and}\quad\int_{\RM}\widehat\psi(x)e^{\gamma x}\,\diff x=+\infty,
\]
which further implies that
\[
\int_{\RM}\widehat{\varrho_T}(x)e^{\gamma x}\,\diff x=-\infty.
\]
By Bochner's Theorem, \(\xi\mapsto\varrho_T(\xi)\) is not positive definite, and so neither is \(l_T\), which concludes the proof.
\end{proof}

Before we end this section, we note that there is a range of values of strong surface tension $T\in(\frac{1}{3},\frac{4}{\pi^2})$ where the kernel $K_T$ is not
monotone. As we will see, this has implications when trying to  
 establish monotonicity of solutions along the supercritical global solution branches described in Section \ref{s:global} below; see Proposition~\ref{prop:leaving_c/2} and the discussion in Section~\ref{sec:discussion} in general.

\subsection{Regularity properties and decay}
In this subsection we split \(K_T\) according to its singularities, and determine the precise regularity of these (there are two of them, both at the origin). We also record the rapid decay and smoothing properties of \(K_T\). Write 
\[
l_T = l_{-\frac{1}{2}} + l_{\frac{3}{2}} + l_\omega,
\]
with \(l_{-\frac{1}{2}}(\xi) = \frac{1}{\sqrt{T \lvert \xi \rvert}}\), \(l_{\frac{3}{2}}(\xi) = \sqrt{\frac{\lvert \xi \rvert}{1 + T \xi^2}} - \frac{1}{\sqrt{T \lvert \xi \rvert}}\) and \(l_\omega(\xi) = l_T(\xi) - \sqrt{\frac{\lvert \xi \rvert}{1 + T \xi^2}}\).
The subscripts represent the regularity of each corresponding term of \(K_T\), as will be seen.  The decay of \(l_{-\frac{1}{2}}(\xi) \eqsim |\xi|^{-\frac{1}{2}}\) for $|\xi|\gg 1$ is clear, and for any fixed \(T>0\), it is readily seen that 
\[
l_\frac{3}{2}(\xi) \eqsim \lvert \xi \rvert^{-\frac{5}{2}}, 
\]
and
\[
l_\omega(\xi) = \sqrt{\frac{\xi}{1 + T \xi^2}} \left( \sqrt{\coth(\xi)} - 1 \right) \eqsim \xi^{-\frac{1}{2}} \, e^{-2\xi},
\]
both for \(|\xi| \gg 1\).

To establish the regularity of the corresponding parts of \(K_T\) we shall use Zygmund spaces. Let $\{\psi_j^2\}_{j = 0}^\infty$ be a partition of unity with \(\psi_0(\xi)\) supported in \(|\xi| \leq 1\), \(\psi_1(\xi)\) supported in \(\frac{1}{2} \leq |\xi| \leq 2\),  and $\psi_j(\xi) = \psi_1(2^{1-j}\xi)$ for \(j \geq 2\). Then the support of each $\psi_j$ is concentrated around $\xi \eqsim 2^j$. With \(D = -i\partial_x\), the Fourier multiplier operators $\psi_j(D) \colon f \mapsto \mathcal{F}^{-1}(\psi_j \hat f)$ characterises the Zygmund spaces: we say \(u \in \zygmund^s(\RM)\) if
	\begin{equation}
		\label{eq:zygmund}
	\| u \|_{\zygmund^s(\RM)} = \sup_j \,2^{js}\, \| \psi_j^2(D) u \|_{L^\infty}
	\end{equation}
is finite. For non-integer values of \(s \geq 0\) the Zygmund spaces coincide with the standard (inhomogeneous) H\"older spaces\footnote{Throughout, we use the notation
that $\NM_0:=\NM\cup\{0\}$.}, 
\begin{equation*}
	\zygmund^s(\RM) \cong C^s(\RM), \qquad s \in \RM_+ \setminus \NM_0,
\end{equation*}
and one furthermore has the embedding \(C^k(\RM) \hookrightarrow \zygmund^k(\RM)\) for integer values of \(k\). We refer the reader to \cite[Section 13.8]{Taylor2011pde} and \cite[Section 1.4]{MR3243741} for further details.

Now, the symbols \(l_{-\frac{1}{2}}\), \(l_{\frac{3}{2}}\) and \(l_\omega\) all have well-defined Fourier transforms, and we let
\begin{align*}
	K_{-\frac{1}{2}}(x) &= \mathcal{F}^{-1}(1/\sqrt{T \lvert \cdot \rvert})(x), \\
	K_{\frac{3}{2}}(x) &= \mathcal{F}^{-1}(l_{\frac{3}{2}})(x), \\
	K_\omega(x) &= \mathcal{F}^{-1}(l_\omega)(x),
\end{align*}
so that
\begin{equation*}
	K_T(x) = \mathcal{F}^{-1}(l_T)(x) = K_{-\frac{1}{2}}(x) + K_{\frac{3}{2}}(x) + K_\omega(x).
\end{equation*}
From Fourier analysis we know that \(\mathcal{F}^{-1}(1/\sqrt{ \lvert \cdot \rvert})(x)=1/\sqrt{ 2\pi\lvert x \rvert}\) and, additionally,  that the exponential
decay of $l_\omega(\xi)$ for $|\xi|\gg 1$ implies that $K_\omega$ is real-analytic by the Paley--Wiener theorem.
The optimal regularity of $K_{\frac{3}{2}}$ follows from the following theorem about the integral kernel \(K_T\).

\begin{theorem}\label{thm:regularity and decay}
The integral kernel \(K_T\) may be written as 
\[
K_T(x) = \frac{1}{\sqrt{2\pi T|x|}} + K_{\frac{3}{2}}(x) + K_\omega(x), 
\]
where the second term belongs to the optimal H\"older class \(C^\frac{3}{2}\) and the third is real-analytic. The singularity of \(K_T\) thus has the characterization
\[
		\lim_{x \rightarrow 0} \sqrt{\lvert x \rvert} \, K_T(x) = \frac{1}{\sqrt{2\pi T}}.
\] Moreover, 
\[
\lvert K_T(x)\rvert \lesssim e^{-\delta \lvert x \rvert}  \qquad \text{ for } |x| > 1,
\] 
with $\delta < \delta^*$ as given in Corollary~\ref{cor:singularityDistance}. As a consequence, \(K_T\in L^1(\RM)\).
\end{theorem}

\begin{proof}
Most of the first claim was established in the preceding discussion: only the regularity of \(K_{\frac{3}{2}}\) remains. We have $\psi^2_j(D)K_{\frac{3}{2}} = \mathcal{F}^{-1}\left(\psi^2_j l_{\frac{3}{2}} \right)$ and, using the $L^1$-norm to estimate the infinity norm, we have that
	\begin{equation*}
		\| \psi_j^2(D) K_{\frac{3}{2}} \|_{L^\infty} \lesssim  \int_{2^{j-2}}^{2^{j}} \lvert l_{\frac{3}{2}}(\xi)\rvert \diff \xi \lesssim  \int_{2^{j-2}}^{2^{j}} \xi^{-\frac{5}{2}} \diff \xi \eqsim  2^{-\frac{3}{2}j}.
	\end{equation*}
Thus
	\begin{equation*}
		\sup_j 2^{\frac{3}{2}j} \, 	\| \psi_j^2(D) K_{\frac{3}{2}} \|_{L^\infty} \lesssim 1,
	\end{equation*}
which proves that $K_{\frac{3}{2}}(x) \in \zygmund^{\frac{3}{2}}(\RM)$. As for the decay rate of \(K_T\), it is a direct consequence of Corollary~\ref{cor:singularityDistance} and 
the Paley--Wiener theorem.
\end{proof}

We conclude this section by recording some mapping properties of the  convolution operator \(L_T = K_T \ast\). Let $\SM$ be the one-dimensional unit sphere of circumference \(2\pi\), and note that the H\"older and Zygmund spaces are straightforward to define on the compact manifold \(\SM\) (these are the \(2\pi\)-periodic functions in the larger spaces \(C^s(\RM)\) and \(\zygmund^s(\RM))\). 

\begin{lemma}\label{lem:compactOperator}
For each \(T> 0\) and each \(s \geq 0\), $L_T$ is a continuous linear mapping \(\zygmund^s(\RM) \to \zygmund^{s+1/2}(\RM)\) and  is hence compact on $\zygmund^s(\SM)$.
\end{lemma}

\begin{proof} 
	Let $u \in \zygmund^s(\SM)$. Using  that $\psi_j^2(D)u = \mathcal{F}^{-1}(\psi_j^2(\xi) \hat{u}(\xi) )$, a straightforward calculation using the boundedness and decay rate of \(l_T \eqsim l_{-\frac{1}{2}}\) for \(|\xi| \gg 1\) shows that $\| \psi_j^2(D) L_T u \|_{L^\infty} \leqslant 2^{-\frac{j}{2}+2} \| \psi_j^2(D) u \|_{L^\infty}$.
	We then have
	\begin{equation*}
	\sup_j \,2^{j(s + \frac{1}{2})}\, \| \psi_j^2(D) L_T u \|_{L^\infty} \lesssim \sup_j \,2^{js}\, \| \psi_j^2(D) u \|_{L^\infty},
	\end{equation*}
which proves the first assertion. Since \(\SM\) is compact it follows that the embedding \(\zygmund^{s+\frac{1}{2}}(\SM) \hookrightarrow \zygmund^{s}(\SM)\) is compact as well, and thus \(L\) is a compact operator on any Zygmund (or H\"older, or \(C^k\)) space defined over \(\SM\).
\end{proof}

\section{One-dimensional bifurcation}\label{sec:onedim}

Since $K\in L^1(\RM)$, it may be periodised to an arbitrary period.  In particular,
given a $2\pi$-periodic $f \in L^\infty(\RM)$ we can define the action of $L_T=K_T\ast$ on $f$ through a convolution of $f$ with a $2\pi$-periodic kernel $K_p$ over a single period:
\begin{align*}
	L_Tf(x) = \int_\RM K_T(x-y) f(y) \diff y &= \int_{-\pi}^{\pi} \left( \sum_{k\in\ZM} K_T(x-y+2k\pi) \right) f(y) \diff y \\
	&=: \int_{-\pi}^{\pi} K_p(x-y) f(y) \diff y.
\end{align*}
Clearly $K_p$ is even, strictly positive on $\RM$ and satisfies $\|K_p\|_{L^1(-\pi,\pi)}=1$.  Further, by Theorem \ref{thm:regularity and decay} we know that
$K_p$ is smooth on $\RM\setminus 2\pi\ZM$,  and that for $T>\frac{4}{\pi^2}$ it follows by Theorem \ref{thm:complete monotonicity} and 
\cite[Proposition 3.2]{Ehrnstrom2016owc} that $K_p$ completely monotone function on the half period $(0,\pi)$.
To find nontrivial solutions of the equation \eqref{e:profile1}, or, equivalently, of \eqref{eq:whitham}, we  fix $s>1/2$ and define a map $F \colon \zygmund^s_\mathrm{even}(\SM)\times\RM\to \zygmund^s_\mathrm{even}(\SM)$
via
\begin{equation}
	\label{e:functional}
	F(u,c)=u-cL_T(u)+L_T(u^2),
\end{equation}
where \(\zygmund^s_\mathrm{even}(\SM)\) is the subspace of even functions in \(\zygmund^s(\SM)\).  Note this map is well-defined
since $\zygmund^s_{\rm even}(\SM)$ is a Banach algebra for any $s>0$.  Then the roots of $F$ correspond to the even and $2\pi$-periodic solutions of \eqref{e:profile1} with wavespeed $c$.  The choice \(s > \frac{1}{2}\) is by convenience, as functions of that regularity have absolutely convergent Fourier series \cite{Katznelson2004ait}.

\begin{figure}[t]\begin{center}
(a)\,\includegraphics[scale=0.44]{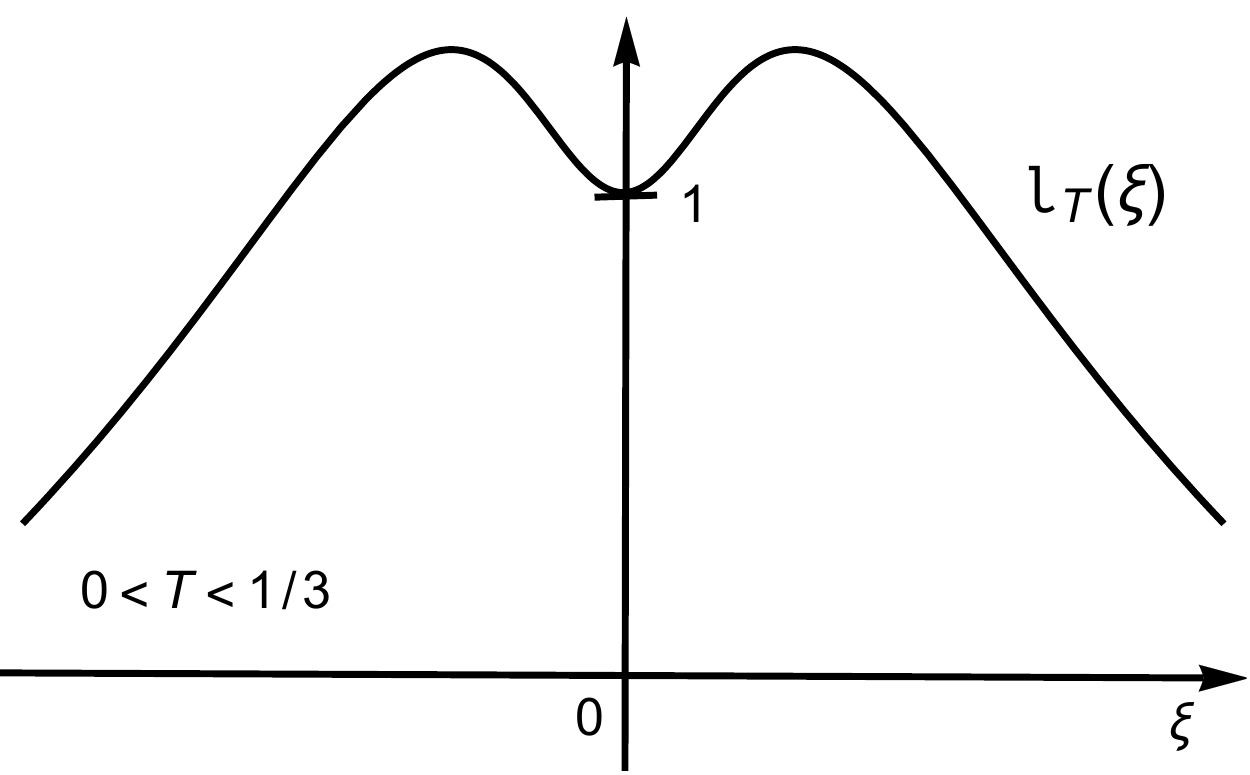}\quad(b)\,\includegraphics[scale=0.44]{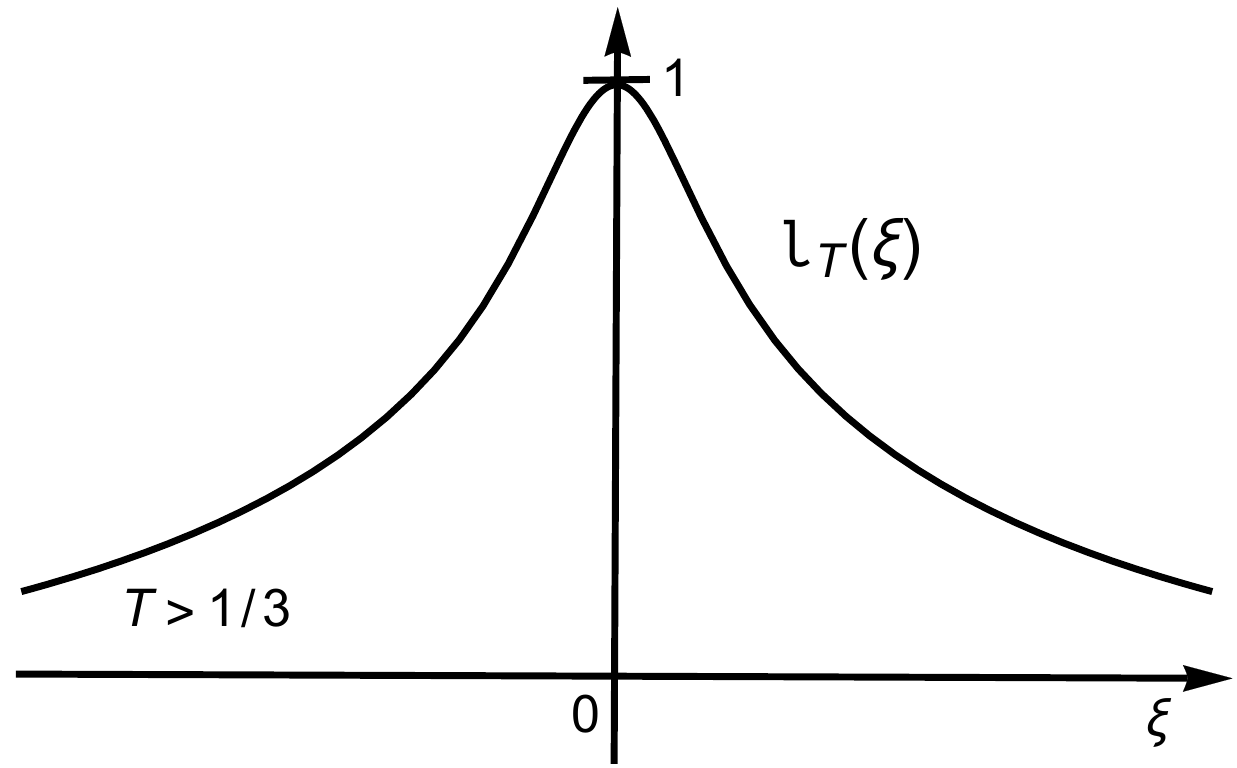}
\caption{Schematic drawings of the behavior of the symbol $l_T(\xi)$ for (a) weak surface tension $0<T<1/3$ and
for (b) strong surface tension $T>1/3$.  In both cases, the symbol is strictly positive and decays  as $|\xi|^{-1/2}$  as $|\xi|\to \infty$.}\label{F:symbol}
\end{center}
\end{figure}

Now, we begin with the observation that $F(0,c)=0$ for all $c\in\RM$ and that the linearised operator
\begin{equation*}
	D_u F[0,c]=\mathrm{Id}-cL_T
\end{equation*}
has a nontrivial kernel in $\zygmund^s_{\rm even}(\SM)$ if and only if $c \, l_T(k) = 1$ for some $k\in\NM_0$ 
(we intentionally include the case \(k=0\) as it will play a role in the two-dimensional bifurcation to come). 
Consequently, for a fixed $c\in\RM$ we have
\begin{equation}\label{eq:nontrivial kernel}
\ker D_u F[0,c] =\mathrm{span}\left\{\cos(kx) \colon k\in\NM_0 \text{ such that } cl_T(k)=1  \right\},
\end{equation}
and hence the multiplicity of the kernel depends sensitively on the graph of the function $l_T(\xi)$.  In particular, if $T>1/3$ then $l_T(\xi)$ is monotone decreasing
on $\RM_+$ and hence the above kernel is simple: see Figure \ref{F:symbol}.  
If $0<T<1/3$, however, the function $l_T$ has exactly one local extremum (a maximum) in the interior of $\RM_+$, 
whence opening the possibility of two different positive integers for which \(l_T(m) = l_T(k)\): again, see Figure \ref{F:symbol}.
A simple calculation shows that for a fixed $k\in\NM_0$, the kernel will be simple if and only if $T\notin\{T_*(n;k)\}_{n\in\NM_0}$, where\footnote{Note that the 
function $T_*(\cdot;\cdot)$ can be extended to the cases $n=0$ and $k=0$ through continuity.}
\[
T_*(n;k):=\frac{n\tanh(k)-k\tanh(n)}{kn\left(n\tanh(n)-k\tanh(k)\right)},
\]
while it will have multiplicity exactly two when  $T = T_*(n;k)$ for some \(n \in \NM_0\).  
Note for each fixed $k$ that the function $T_*(\cdot;k)$ is strictly decreasing
on $\NM_0$ with $T_*(n;k)\to 0$ as $n\to\infty$.  Furthermore, the quantity $\max_{n\in\NM_0}T_*(n;k)=T_*(0;k)$ is a strictly decreasing function of $k$ on $\NM_0$
tending to zero as $k\to\infty$.  

Throughout the remainder of this section, we turn our attention to the branches of solutions \(\{(u,c)\}\) bifurcating from the trivial line $u=0$ at some wavespeed  $c_*$ for a fixed value of the surface tension \(T > 0\) and where $\ker D_u F[0,c_*]$ is one-dimensional; two-dimensional bifurcation in the case \(0 < T < \frac{1}{3}\) is dealt with in Section~\ref{sec:twodim}. Note that while one-dimensional kernels appear both for sub- and supercritical wave speeds, separated by \(c = 1\), two-dimensional kernels only appear for \(c \in (0,1]\): see Section \ref{sec:twodim} below.

\subsection{The parameters}
To investigate the bifurcations we will make use in the following sections of three positive quantities --- the wavespeed $c$, the surface tension $T$, and a scaling in the period of the waves, $\kappa$.
While the first two appear directly in the steady problem \eqref{e:profile1}, the scaling \(\xi \mapsto \kappa \xi\) is realised by introducing the corresponding dependence in the convolution operator \(L\),  so that
\begin{equation}
	\label{eq:operatorScaling}
	\widehat{L_{\kappa,T}}(\xi) = l_{\kappa,T}(\xi) := l_T (\kappa \xi).
\end{equation}
This operator agrees with the original one for $\kappa = 1$. In particular, finding $2\pi$-periodic solutions of \eqref{e:profile1} with symbol $L_{\kappa,T}$ is equivalent to finding
$2\pi/\kappa$-periodic solutions of \eqref{e:profile1} with symbol $L_T=L_{1,T}$.
This allows us to treat different wavelengths in the same equation by moving the wavelength parameter to \(L_{\kappa,T}\).  In what follows, we will thus modify
\eqref{e:functional} and seek non-trivial solutions of the map
\begin{equation}\label{e:functional_rescaled}
F_\kappa(u,c)=u-cL_{\kappa,T}(u)+L_{\kappa,T}\left(u^2\right)
\end{equation}
in $\zygmund^s_{\rm even}(\mathbb{S})\times\RM$ for a fixed $\kappa>0$.

Since surface tension is a property of the medium, while the speed and wavenumber are properties of particular waves, it is physically more relevant to use the two latter as bifurcation parameters, while holding the surface tension fixed. This is what we will do in the following.

\subsection{Local bifurcation via Lyapunov--Schmidt}\label{subsec:local}
The following theorem establishes, for fixed wavelength and surface tension, the local bifurcation of small amplitude steady solutions the capillary-gravity 
Whitham equation \eqref{e:w1}. Although this is by now a standard Crandall--Rabinowitz type result \cite{MR2004250}, we prove the result using a direct Lyapunov--Schmidt reduction as to prepare for the two-dimensional bifurcation in Section~\ref{sec:twodim}. This is similar to the strategy in \cite{eew10}. As \(\kappa\) and \(T\) will be fixed --- assuming that we already have a one-dimensional kernel as described in the beginning of this section --- we shall here suppress the dependence upon these parameters. 

\begin{theorem}\label{thm:1Dlocal}
Let \(k\in\NM\) and set $c_0 = l_{\kappa,T}(k)^{-1}$. For any $T, \kappa >0$ such that $\mathrm{dim} \ker D_u F_\kappa(0, c_0) = 1$ there exists a smooth curve
	\begin{equation*}
		\{ \left( u(t), c(t) \right)  \colon 0 < |t|\ll 1 \}
	\end{equation*}
	of small-amplitude, $2\pi$-periodic even solutions of the steady capillary-gravity Whitham equation \eqref{e:profile1} with symbol given by \eqref{eq:operatorScaling}. These solutions satisfy
	\begin{align*}
		u(t) &= t\cos(kx) + \mathcal{O}(t^2) \\
		c(t) &= c_0 + \mathcal{O}(t).
	\end{align*}
in $\zygmund^s_\mathrm{even}(\SM) \times \RM$, and constitute all nontrivial solutions in a neighbourhood of $(0,c_0)$ in that space.
\end{theorem}

\begin{remark}\label{rem:transcritical}
There is an additional but qualitatively different bifurcation taking place at \(c =1\), where the straight curve of constant solutions \((u,c) = (c-1,c)\) crosses the trivial solution curve \((0,c)\). These solutions must be taken into consideration when constructing non-constant waves at \(c=1\) when the kernel is two-dimensional, see Theorem~\ref{thm:2DLocal}. 
\end{remark}

\begin{remark}\label{rem:kappa}
By considering the role of \(\kappa\) in the proof of Theorem~\ref{thm:1Dlocal} one can see that by varying \(\kappa\) one obtains a one-dimensional family of solution curves, the starting points of which depend smoothly on $\kappa$. This may be seen also by applying the implicit function theorem directly to \ref{e:functional}. For each \(k \in \NM\) we thus obtain a two-dimensional sheet of solutions, 
\begin{equation}\label{eq:S^k}
\mathcal{S}^k = \left\{ (u(t,\kappa), c(t,\kappa), \kappa) \colon 0 < |t| \ll 1, |\kappa - \kappa_0| \ll 1\right\} 
\end{equation}
parameterised by \((t,\kappa)\) in a neighbourhood of a bifurcation point \((0,\kappa_0)\).
\end{remark}

\begin{proof}
As stated above, we suppress the dependence on the fixed parameters $T$ and $\kappa$ throughout.
According to the assumptions and the discussion after \eqref{eq:nontrivial kernel}, on $\zygmund^s_{\rm even}(\SM)$ we have   
\[
\kernel D_u F(0,c_0) = \ker (\mathrm{Id} - c_0 L) = \spn\{\cos(k \cdot)\}.
\]	
We first write
	\begin{align*}
		u(t) &= t \cos(k x) + v(t), \\
		c(t) &= c_0 + r(t),
	\end{align*}
	with $v(t) \in \zygmund^s_\mathrm{even}(\SM)$ such that $\int_{-\pi}^\pi \cos(k x) v \diff x = 0$ and $r(t) \in \RM$, and proceed to show the existence of $v$ and $r$ such that for $|t|\ll 1$ we have
	\begin{equation}
		\label{e:F1D=0}
		F(t \cos(k x) + v(t), c_0 + r(t)) = 0.
	\end{equation}
As a subspace of \(L^2(\SM)\), we equip $\zygmund^s_\mathrm{even}(\SM)$ with the $L^2$ inner product \(\langle f, g \rangle = \frac{1}{\pi} \int_{-\pi}^{\pi} fg \, \diff x\)  and let $\Pi\colon \zygmund^s_\mathrm{even}(\SM) \rightarrow \kernel D_u F(0,c_0)$ be the projection onto $\spn \{\cos(k\cdot)\}$ parallel to $\ran(D_u F(0,c_0))$. 
Since  \(D_u F(0,c_0)\) is a symmetric Fredholm operator with index \(0\) by Corollary \ref{cor:fredholm} below, it follows that 
$\zygmund^s_\mathrm{even}(\SM)$ may be decomposed as a direct sum between its kernel and range. In particular, \eqref{e:F1D=0} is equivalent to the system of equations
	\begin{equation}
		\label{e:proj}
		\begin{aligned}
			\Pi F(t \cos(k x) + v, c_0 + r) &= 0, \\
			(I-\Pi) F(t \cos(k x) + v, c_0 + r) &= 0,
		\end{aligned}
	\end{equation}
	where we have suppressed the \(t\)-dependence in \(v\) and \(r\). Noting that
	\begin{align*}
		&F(t \cos(k x) + v, c_0 + r) \\
		&= t \cos(k x) + v - (c_0+r)L (t \cos(k x)+v) + L (t \cos(k x) + v)^2  \\
		&= D_u F(0,c_0) (v +  t\cos(k x))\\ 
		&\quad - r L (t \cos(k x) + v) + L (t \cos(k x) + v)^2,
	\end{align*}
and that \(\cos(k\cdot)\) is in the kernel of \(D_u F(0,c_0)\), the equation \eqref{e:F1D=0} may be rewritten as
	\begin{equation}
		\label{eq:equiv1D}
		D_uF(0,c_0) v =  r L (t \cos(k x) + v) - L (t \cos(k x) + v)^2 =: g(t, r, v)
	\end{equation}
and hence, recalling that $v\in(1-\Pi)\zygmund^s_{\rm even}(\SM)$, \eqref{e:proj} is equivalent to the system
	\begin{equation}
		\label{e:proj2}
		\begin{aligned}
			0 &= \Pi g(t, r, v) \\
	D_uF(0,c_0) v &= (\mathrm{Id} - \Pi)g(t, r, v).
		\end{aligned}
	\end{equation}
Finally, observe  that since $D_uF(0,c_0)$ is invertible on $(I-\Pi) \zygmund^s_\mathrm{even}(\SM)$, the second equation in \eqref{e:proj2} can be rewritten as 
	\begin{equation*}
		v = [D_uF(0,c_0)]^{-1} (\mathrm{Id} -\Pi) g(t, r, v).
	\end{equation*}
Concerning this latter equation, note that at $(t,r) = (0,0)$  we have both that $v=0$ is a solution and that the Frech\`et derivative with respect to $v$ is invertible on $(\mathrm{Id}-\Pi)\, \zygmund^s_\mathrm{even}(\SM)$ (because \(D_uF(0,c_0)\) is). 
Therefore, by the implicit function theorem on Banach spaces, the second line of \eqref{e:proj2} has a unique solution $v(t, r) \in (\mathrm{Id}-\Pi)\, \zygmund^s_\mathrm{even}(\SM)$ defined in a neighbourhood of $(t,r) = (0, 0)$, and depending analytically on its arguments. By uniqueness, $v(0,r)=0$ for all \(|r| \ll 1\). Moreover, differentiation with respect to \(t\) at $(t,r) = (0,0)$ in \eqref{eq:equiv1D} shows that $\frac{\partial}{\partial t}v(0,r) = 0$, which implies that $v$ has no constant or linear terms in $t$. As it is smooth in \(t\), it may be expanded in an (at least) quadratic series around \(t=0\).

We now need to solve the equation
	\begin{equation*}
	\Pi g(t, r, v(t, r)) = Q(r,t) \cos(kx) = 0
	\end{equation*}
	for \(r\), with
	\begin{equation*}
	Q(t,r) := \langle g(t, r, v(t, r)) , \cos(k \cdot) \rangle.
	\end{equation*}
	Notice that that $Q(0,r)=0$ since $v(0,r)=0$ for all \(r\), which together with the symmetry of $L$ implies that we can write
	\begin{equation*}
		 Q(t, r) = t \left[r \, l(k) + R(t,r)\right],
	\end{equation*}
	where $R$ is analytic with $R(0,0) = \partial_r R(0,0) = 0$, again due to the properties of $v$ (here, \(l = l_{T, \kappa}\)). An application of the implicit function theorem to the equation \(r \, l(k) \pi + R(t,r) = 0\) at \((t,r) = (0,0)\) then yields the existence of a locally unique smooth function $r \colon t \mapsto r(t)$ with \(r(0) = 0\) such that 

\[
Q(t, r(t))) = t ( r(t)\, l(k) + \tilde{R}(t, r(t))) = 0
\]
for all \(|t|\ll1\). This concludes the proof.
\end{proof}

\subsection{Global bifurcation (analytic)}\label{s:global}
We now extend the local bifurcation curves from Section~\ref{subsec:local} to global ones by the means of the analytic bifurcation theory 
pioneered by Dancer \cite{Dancer1973btf, Dancer1973gso} and then developed further by Buffoni and Toland \cite{BuffoniToland}. 
For fixed $s>1/2$, we define $N \colon \zygmund^s_{\rm even}(\SM)\times\RM\to \zygmund^{s+1/2}_{\rm even}(\SM)$ by
\[
N(u,c) = L(cu-u^2).
\]
Fixed points of \(N\) are solutions of the steady capillary-gravity Whitham equation \eqref{e:profile1}, and conversely.  Let 
\[
S =\left\{(u,c)\in \zygmund^s_{\rm even}(\SM)\times\RM \colon F(u,c)=0\right\}
\] 
be the set of solutions (fixed points of \(N\)). 
Note that Lemma~\ref{lem:compactOperator} implies that \(S \subset \zygmund^\infty_{\rm even} \times \RM\), so that all solutions are smooth: for 
details, see Proposition~\ref{prop:smooth} below. 
By combining this with a diagonal argument one obtains the following compactness result.

\begin{lemma}\label{lem:compactness}
Bounded and closed sets in $S$ are compact in $\zygmund^s_{\rm even}(\SM)\times\RM$.
\end{lemma}

\begin{proof}
Let $K\subset S\subset \zygmund^s_{\rm even}(\SM)\times\RM$ be closed and bounded, and pick a sequence $(u_j, c_j)_j \subset K$. 
Since $\{c \in \RM \colon (u,c)\in K\}$ is a closed and bounded subset of $\RM$, it is compact. This means that $(c_j)_j$ has a convergent subsequence, name it $(c_k)_k$.
As the map 
\[
\zygmund^s_{\rm even}(\mathbb{S})\times\RM\ni(u,c)\mapsto cu-u^2\in\zygmund_{\rm even}^s(\SM)
\]
is continuous for $s>1/2$, and since the map $L$ is compact on $\zygmund^s_{\rm even}(\SM)$ thanks to Lemma \ref{lem:compactOperator},
it follows that after passing to a further subsequence $(u_l,c_l)_l\subset K$ that $(N(u_l,c_l))_l$  converges in $\zygmund^s_{\rm even}(\SM)$ to some function $u$.
Since $u_l=N(u_l,c_l)$ by definition, passing to limits implies the sequence $(u_l,c_l)_l$ converges in $\zygmund^s_{\rm even}(\SM)\times\RM$ with
limit $(u,c)\in S$.  As $K$ is closed it follows that $(u,c)\in K$, establishing that $K$ is compact.
\end{proof}

\begin{corollary}\label{cor:fredholm}
The Frech\`et derivative $D_u F(u,c)$ is a Fredholm operator of index $0$ at any point $(u,c) \in C^s_\mathrm{even}(\SM)\times \RM$. 
\end{corollary}
\begin{proof}
This follows immediately from Lemma~\ref{lem:compactness} as then
	\begin{equation*}
		D_u F(u,c) = \mathrm{Id}-L(c - 2 u)
	\end{equation*}
	is a compact perturbation of the identity.
\end{proof}

\begin{theorem}\label{thm:1Dglobal}
Whenever
\begin{align}\label{eq:global condition}
\frac{3c_0 l(2k) - l(2k) - 2}{(c_0-1) (c_0 l(2k) - 1)}
\end{align}
is finite and non-vanishing the local bifurcation curve $t \mapsto \left( u(t), c(t) \right)$, \(|t| \ll 1\), from Lemma~\ref{thm:1Dlocal} extends to a continuous and locally analytically 
re-parameterisable curve in $\zygmund^s_{\rm even}(\SM)\times\RM$ defined for all $t \in [0, \infty)$. One of the following alternatives holds:
	\begin{enumerate}
		\item[(i)] $\|(u(t), c(t)) \|_{\zygmund^s(\SM) \times \RM} \rightarrow \infty$ as $t \rightarrow \infty$.
		\item[(ii)] $t \mapsto \left( u(t), c(t) \right)$ is $P$-periodic for some finite \(P\), so that the curve forms a loop.
	\end{enumerate}
\end{theorem}

\begin{remark}\label{rem:cdotdot}
We note that
\[
\ddot{c}(0;k)=\left\{\begin{aligned}
									&\frac{10}{(3T-1)k^2}+\mathcal{O}(1)\quad{\rm for}~|k|\ll 1\\
									&-(\sqrt{2}-1)(Tk)^{-1/2}+\mathcal{O}\left(k^{-1}\right)\quad{\rm for} ~k\gg 1.
									\end{aligned}\right.
\]
For $T>1/3$ it follows that $(0,c_0)$ undergoes a supercritical pitchform bifurcation for small $k$, and a subcritical pitchfork bifurcation for large $k$.
Note numerically, we observe there exists a unique $k_*=k_*(T)>0$ such that $\ddot{c}(0)>0$ for $0<k<k_*$ and $\ddot{c}(0)<0$ for $k>k_*$.
For $0<T<1/3$, both the numerator and denomenator of \eqref{eq:global condition} change signs.
Note that one may be able to do global bifurcation when $\ddot{c}(0)=0$ but inspecting $c^{(4)}(0)$: see, for example, \cite[Theorem 6.1]{Ehrnstrom2016owc}.
We do not pursue this here.
\end{remark}

\begin{proof}
This theorem is a version of the global analytic bifurcation theorem in \cite{BuffoniToland}, and --- apart from the bifurcation formulas --- the proof goes as in the purely gravitation case in \cite{Ehrnstrom2013gbf, Ehrnstrom2016owc}. The assumptions are fulfilled from Lemma~\ref{lem:compactness} and Corollary~\ref{cor:fredholm} if one can just show that some derivative \(c^{(k)}(0)\) is non-vanishing. We give the calculations for \(\dot{c}(0)\) and \(\ddot{c}(0)\) in the Appendix; the first is \(0\), and the second is given by \eqref{eq:global condition}.  Note that a third alternative in the theorem in \cite{BuffoniToland} does not happen here, as the set \(\zygmund^s_{\rm even}(\SM) \times \RM\) lacks a boundary.
\end{proof}

There are a few more things one can say about the global bifurcation curves, both numerically and analytically, and we discuss the global bifurcation diagram in detail in Section~\ref{sec:discussion}. In particular, the cases of strong and weak surface tension are summarised in Figures~\ref{fig:strong} and~\ref{fig:weak}, respectively.

\section{Two-dimensional local bifurcation}\label{sec:twodim}
We now focus our attention on the case of a two-dimensional bifurcation kernel in \(\zygmund^s_\text{even}(\SM)\). To enable the necessary two degrees of freedom we shall make use of the wavelength \(\kappa\) in addition to the wavespeed \(c\), while the surface tension \(T\) is assumed to be fixed. We shall therefore study for $\kappa>0$ the operator
\[
F_{\kappa}(u,c) = u + L_{\kappa}(u^2 - c u)
\]
on $\zygmund^s_{\rm even}(\SM)\times\RM$, along with its linearisation
\[
\Lop = D_u F_{\kappa_0}(0, c_0) = \mathrm{Id} - c_0 L_{\kappa_0},
\]
assuming that \(T, \kappa_0, c_0 > 0\) are constants such that
\begin{equation}\label{eq:twodim kernel}
\ker(\Lop) = \mathrm{span}\{\cos(k_1 \cdot), \cos(k_2 \cdot)\},
\end{equation}
which happens when $\kappa_0, c_0 > 0$ and $k_1, k_2 \in \NM_0$, $k_1 \ne k_2$, are such that 
\[
c_0 = l_{\kappa_0}(k_1)^{-1} = l_{\kappa_0}(k_2)^{-1}, 
\]
as described at the start of Section~\ref{sec:onedim} (we suppress the dependence on \(T\), as it will not be used apart from in this assumption). A two-dimensional kernel can arise only for \(c_0 \in (0,1]\). Let now \(1 \leq k_1 \leq k_2\). With $\mathcal{S}^k$ being the sheet of $2\pi/k$-periodic solutions defined in \eqref{eq:S^k} we shall show that in addition to the solutions in $\mathcal{S}^{k_1}$ and $\mathcal{S}^{k_2}$, we may obtain solutions in a set called $\mathcal{S}^{mixed}$ consisting of perturbations of functions in the span of \(\cos(k_1 \cdot)\) and \(\cos(k_2 \cdot)\).
Assuming that \(k_1 \leq k_2\), the resonant case when $k_2$ is an integer multiple of $k_1$ (sometimes referred to as Wilton ripples) is more difficult than the generic case, but we follow here the procedure in \cite{eew10, Ehrnstrom2015tsw} to construct a slit disk of solutions also in that case. Numerical calculations indicate that this set is optimal \cite{Remonato2017nbf}. 

When one of the wavenumbers is zero (meaning \(c_0 = 1\)), we instead call that one \(k_2\), and we will automatically have the resonant case, as then \(k_1 \mid k_2\). That case is included in the below theorem. Hence, at \(c=1\) there is a nontrivial bifurcation, but the arising waves always have a non-zero component in the constant direction.

\begin{theorem}\label{thm:2DLocal}
Let $T>0$ be fixed and assume that \eqref{eq:twodim kernel} holds for some distinct $k_1,k_2\in\NM_0$.
	\begin{enumerate}
		\item[(i)] When $k_1$ does not divide $k_2$ there is a full, smooth, sheet 
			\begin{equation*}
				\mathcal{S}^{mixed} = \{ \left(u(t_1, t_2), c(t_1, t_2), \kappa(t_1, t_2)\right) \colon 0 < |(t_1, t_2)|\ll 1 \}
			\end{equation*}
			of solutions in $\zygmund^s_{\rm even}(\SM)\times\RM\times\RM_+$ of the form
			\begin{align*}
				u(t_1, t_2) &= t_1 \cos(k_1 x) + t_2 \cos(k_2 x) + \mathcal{O}(|(t_1, t_2)|^2), \\
				c(t_1, t_2) &= c_0 + \mathcal{O}((t_1, t_2)), \\
				\kappa(t_1, t_2) &= \kappa_0 + \mathcal{O}((t_1, t_2)),
			\end{align*}
to the steady capillary-gravity Whitham equation \eqref{e:profile1}. The set $\mathcal{S}^{k_1} \cup \mathcal{S}^{k_2} \cup \mathcal{S}^{mixed}$ contains all 
nontrivial solutions in $\zygmund^s_{\rm even}(\SM)\times\RM\times\RM_+$ of this equation in a neighbourhood of $(0, c_0, \kappa_0)$.\\

		\item[(ii)] When $k_1$ divides \(k_2\) there exists for any \(\delta > 0\) a small but positive \(\varepsilon_\delta\) and a slit, smooth, sheet 
		\begin{equation*}
		\mathcal{S}^{mixed}_\delta = \{ \left(u(\varrho, \vartheta), c(\varrho, \vartheta), \kappa(\varrho, \vartheta)\right) \colon 0<\varrho<\varepsilon_\delta,\, \delta < |\vartheta| < \pi - \delta  \}
		\end{equation*}
of solutions in $\zygmund^s_{\rm even}(\SM)\times\RM\times\RM_+$ of the form
		\begin{align*}
		 u(\varrho, \vartheta) &= \varrho\cos(\vartheta) \cos(k_1 x) + \varrho\sin(\vartheta) \cos(k_2 x) + \mathcal{O}(\varrho^2), \\
		c(\varrho, \vartheta) &= c_0 + \mathcal{O}(\varrho), \\
		\kappa(\varrho, \vartheta) &= \kappa_0 + \mathcal{O}(\varrho).
		\end{align*}
to the steady capillary-gravity Whitham equation \eqref{e:profile1}.  In a neighbourhood of $(0, c_0, \kappa_0)$,
the set $\mathcal{S} = \mathcal{S}^{k_2} \cup \mathcal{S}^{mixed}_\delta$ contains all nontrivial solutions in $\zygmund^s_{\rm even}(\SM)\times\RM\times\RM_+$
of \eqref{e:profile1} such that $\delta < |\vartheta| < \pi-\delta$.
	\end{enumerate}
\end{theorem}

\begin{remark}
The order of vanishing of the functions $c-c_0$ and $\kappa-\kappa_0$ in Theorem \ref{thm:2DLocal} is analyzed in Section \ref{S:2d_der_calc} of Appendix~\ref{app:A}.  
\end{remark}

\begin{remark}
The bifurcation theorem Theorem \ref{thm:2DLocal} shows that near a two-dimensional bifurcation point in the case where $k_2/k_1 \notin \NM_0$ there exists a full disk of solutions (for fixed $\kappa$), while if $k_2/k_1 \in \NM_0$ the disk is slit with one axis removed. This situation is summarised in Figure \ref{fig:solutionDisk}. In particular this means that it is possible to find curves connecting solutions with different wavenumbers, consistent with the recent numerical findings in \cite{Remonato2017nbf}.
\end{remark}

\begin{figure}[t]
	\centering
	
	\includegraphics[width=0.7\linewidth]{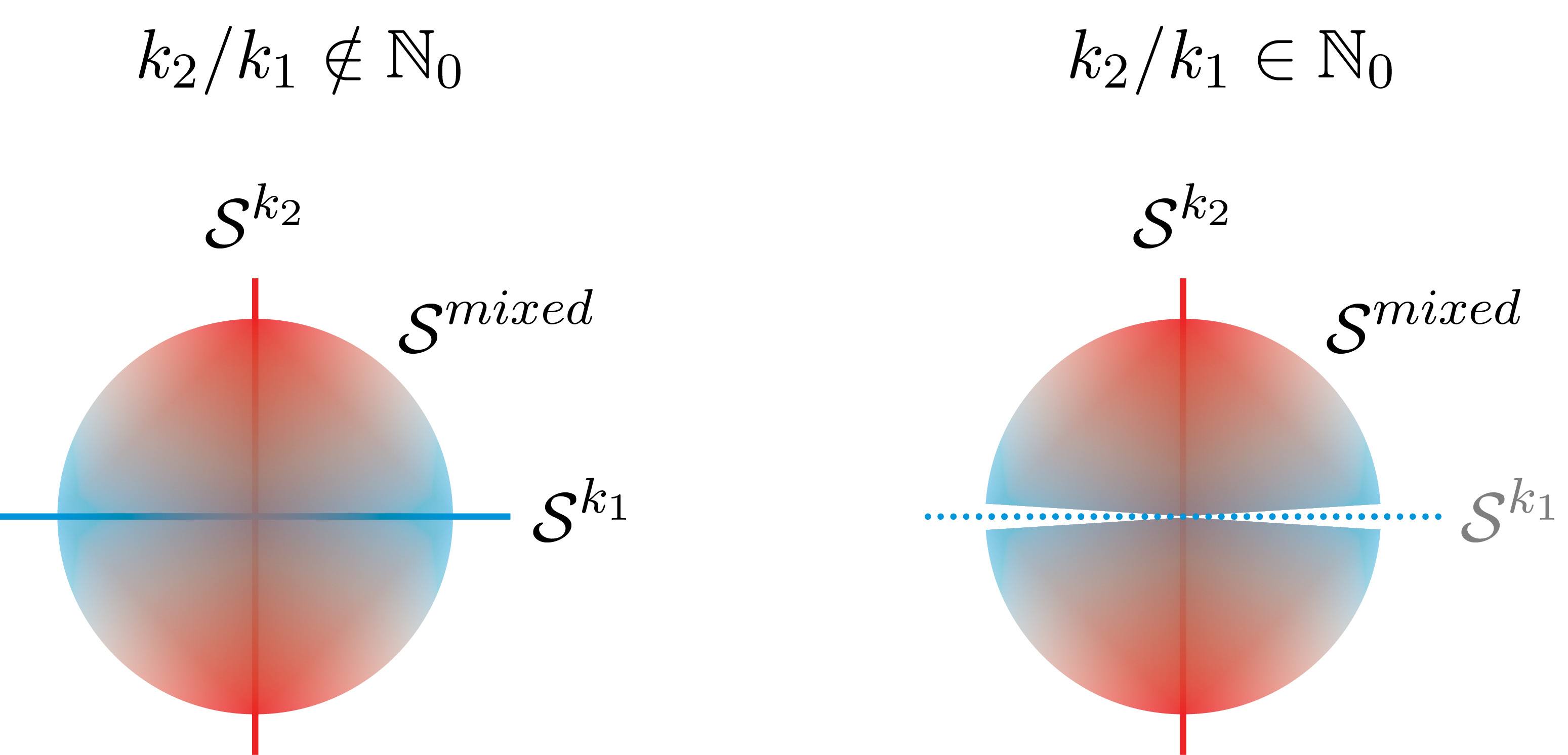}

	\caption{The local solution disks for the steady capillary-gravity Whitham equation \eqref{e:profile1} 
	around a point where the bifurcation kernel is two-dimensional.  The left-hand drawing depicts the situation in Theorem~\ref{thm:2DLocal} (i), whereas the right-hand drawing refers to case (ii) of the same theorem. The blue and red colours represent the proximity of the solutions to the pure \(k_1\)- and \(k_2\)-modes, respectively. In particular, when \(k_1\) divides \(k_2\) we have not found any waves bifurcating in the direction of  \(\cos(k_1 \cdot)\).}
	\label{fig:solutionDisk}
\end{figure}

\begin{proof}
We start by writing
	\begin{align*}
	u(t_1, t_2) &= t_1 \cos(k_1 x) + t_2 \cos(k_2 x) + v, \\
	c(t_1, t_2) &= c_0 + r, \\
	\kappa(t_1, t_2) &= \kappa_0 + p,
	\end{align*}
where, generically, we want to find \(v\), \(r\) and \(p\) parameterised by \((t_1, t_2)\) such that 
	\begin{equation}
	\label{e:F2D=0}
	F_{\kappa_0 + p}(t_1 \cos(k_1 x) + t_2 \cos(k_2 x) + v, c_0 + r) = 0,
	\end{equation}
	for sufficiently small values of \((t_1, t_2)\). As in the proof of Theorem~\ref{thm:1Dlocal}, we let $\Pi\colon C_\mathrm{even}^s(\SM) \rightarrow \ker(D_u F_{\kappa_0}(0,c_0))$ be the projection onto $\ker(D_u F_{\kappa_0}(0,c_0))$ parallel to $\ran(D_u F_{\kappa_0}(0,c_0))$, where we have equipped	 $C_\mathrm{even}^s(\SM)$ with the $L^2$ inner product \(\langle f, g \rangle = \frac{1}{\pi} \int_{-\pi}^{\pi} fg \, \diff x\). According to Corollary~\ref{cor:fredholm} equation \eqref{e:F2D=0} is then equivalent to
	\begin{equation}
		\label{e:proj2D}
		\begin{cases}
			\Pi F_{\kappa(t_1, t_2)}\left( u(t_1, t_2), c(t_1, t_2)\right) = 0 \\
			(\mathrm{Id}-\Pi)F_{\kappa(t_1, t_2)}\left( u(t_1, t_2), c(t_1, t_2)\right) = 0.
		\end{cases}
	\end{equation}
Note that under the above ansatz, where it is assumed that \(\Pi v = 0\),
	\begin{align*}
		F_\kappa\left(u, c\right) &= t_1\cos(k_1 x) + t_2\cos(k_2 x) + v\\ 
		&\quad + L_{\kappa_0+p} \left[ (t_1\cos(k_1 x) + t_2\cos(k_2 x)  + v)^2 \right. \\
			& \left. \quad - (c_0 + r)\,(t_1\cos(k_1 x) + t_2\cos(k_2 x) + v) \right] \\
			&= (v - c_0 L_{\kappa_0+p}v) + t_1 \left( \cos(k_1 x) - c_0 L_{\kappa_0+p} \cos(k_1 x) \right) \\
			&\quad + t_2 \left( \cos(k_2 x) - c_0 L_{\kappa_0+p} \cos(k_2 x) \right)\\ 
			&\quad - r L_{\kappa_0+p} \left( t_1 \cos(k_1 x) + t_2 \cos(k_2 x) + v \right) \\
			&\quad + L_{\kappa_0+p} \left( t_1 \cos(k_1 x) + t_2 \cos(k_2 x) + v \right)^2,
	\end{align*}
	and writing $L_{\kappa_0+p} = L_{\kappa_0} + (L_{\kappa_0+p} - L_{\kappa_0})$ we have
	\begin{align*}
		F_\kappa\left(u, c\right) &= D_u F_{\kappa_0}(0,c_0) v - c_0 (L_{\kappa_0+p} - L_{\kappa_0}) v\\ 
		&\quad - t_1 c_0 (L_{\kappa_0+p} - L_{\kappa_0}) \cos(k_1 x) - t_2 c_0 (L_{\kappa_0+p} - L_{\kappa_0}) \cos(k_2 x)\\ 
		&\quad - r L_{\kappa_0+p} \left( t_1 \cos(k_1 x) + t_2 \cos(k_2 x) + v \right)\\
	&\quad  + L_{\kappa_0+p} \left( t_1 \cos(k_1 x) + t_2 \cos(k_2 x) + v \right)^2\\[3pt]
	&=: D_u F_{\kappa_0}(0,c_0) v - g(t_1, t_2, r, p, v).
	\end{align*}
	Therefore \eqref{e:F2D=0} is equivalent to
	\begin{equation}
		\label{eq:equivForm}
		D_u F_{\kappa_0}(0,c_0) v = g(t_1, t_2, r, p, v),
	\end{equation}
	and we can rewrite \eqref{e:proj2D} as
	\begin{equation}
		\label{e:proj2D2}
		\begin{cases}
			0 = \Pi g(t_1, t_2, r, p, v) \\
			D_u F_{\kappa_0}(0,c_0) v = (\mathrm{Id}-\Pi) g(t_1, t_2, r, p, v).
		\end{cases}
	\end{equation}
	Note that since $v$ is orthogonal to $\ker(D_u F_{\kappa_0}(0,c_0))$ the second equation in \eqref{e:proj2D2} reads $v = D_u F_{\kappa_0}(0,c_0)^{-1}(\mathrm{Id}-\Pi)g(t_1, t_2, r, p, v)$.
	It is clear that 
\[
D_u F_{\kappa_0}(0,c_0) v - (\mathrm{Id}-\Pi) g(t_1, t_2, r, p, v) = 0
\] 
has the solution $(t_1, t_2, r, p, v) = (0, 0, 0, 0, 0)$ and at that point the Frech\`et derivative respect to $v$ is 
$D_u F_{\kappa_0}(0,c_0)$, which is invertible on $(\mathrm{Id}-\Pi)C_\mathrm{even}^s(\SM)$.
	The implicit function theorem then ensures the existence of a solution $v = v(t_1, t_2, r, p) \in (\mathrm{Id}-\Pi)C_\mathrm{even}^s(\SM)$.
By uniqueness we have that \(v(0,0,r,p) = 0\)  for all small enough values of  \(r\) and \(p\).
	Moreover, note that $\frac{\partial}{\partial t_1}v(0,0,0,0) = 0$ and $\frac{\partial}{\partial t_2}v(0,0,0,0) = 0$. This follows by differentiating \eqref{eq:equivForm} respect to $t_1$ or $t_2$, and evaluating at $(t_1, t_2, r, p) = (0, 0, 0, 0)$ recalling that $D_u F_{\kappa_0}(0,c_0)$ is invertible on its range. As a consequence, $v$ depends at least quadratically on $t_1$ and $t_2$.

	We are now left with solving the finite-dimensional problem given by the first equation in \eqref{e:proj2D2}.
	To this end, we decompose the projection $\Pi$  as $\Pi = \Pi_1 + \Pi_2$, where $\Pi_1$ is the projection onto $\cos(k_1 \cdot)$, and $\Pi_2$ is the projection onto $\cos(k_2 \cdot)$. Then
	\begin{equation*}
		\Pi g = \Pi_1 g + \Pi_2 g = Q_1 \cos(k_1 x) + Q_2 \cos(k_2 x),
	\end{equation*}
	with \(Q_j = \langle g, \cos(k_j \cdot) \rangle\), and the first line of \eqref{e:proj2D2} is equivalent to showing that 
	\begin{equation}\label{e:proj2D3}
			Q_1 = Q_2 = 0.
	\end{equation}
	To solve \eqref{e:proj2D3} we consider two cases.

\

\noindent
\emph{The non-resonant case}.	Assume that $k_2/k_1 \notin \NM_0$. Using the properties of $v$ and $\Pi_1$, a direct calculation shows that

\begin{equation}
	\label{e:Q1}
	\begin{aligned}
		Q_1 &= t_1 \left[ c_0 \big(l((\kappa_0+p)k_1) - l(\kappa_0 k_1)\big) + r \, l((\kappa_0+p)k_1) \right]  \\
		& \quad - l((\kappa_0+p)k_1) \left\langle \cos(k_1 \cdot) , \left( t_1 \cos(k_1 \cdot) + t_2 \cos(k_2 \cdot) + v(t_1, t_2, r, p) \right)^2 \right\rangle.
	\end{aligned}
\end{equation}

As $v(0,t_2,r,p)$ is $2\pi/k_2$-periodic and $k_2 \neq k_1$, the above inner term product vanishes for $t_1 = 0$.
Therefore we may write 
	\begin{equation}
		\label{e:Q1Factored}
		Q_1(t_1, t_2, r, p) = t_1 \, \Psi_1 (t_1, t_2, r, p)
	\end{equation}
	with
	\begin{equation}
		\label{e:Psi1}
		\Psi_1(t_1,t_2,r,p) = \int_0^1 \frac{\partial Q_1}{\partial t_1}(z t_1, t_2, r, p) \;\mathrm{d}z,
	\end{equation}
	and note \eqref{e:Q1} implies
	\begin{equation}
		\label{e:Psi1At00}
		\Psi_1(0, 0, r, p) = c_0 \left[ l((\kappa_0+p)k_1) - l(\kappa_0 k_1)\right] + r\, l((\kappa_0+p)k_1).
	\end{equation}
	Similarly, we have

	\begin{equation}
		\label{e:Q2}
		\begin{aligned}
			Q_2 &= t_2 \left[ c_0 \big(l((\kappa_0+p)k_2) - l(\kappa_0 k_2)\big) + r \, l((\kappa_0+p)k_2) \right]  \\
				& \quad - l((\kappa_0+p)k_2) \left\langle \cos(k_2 \cdot) , \left( t_1 \cos(k_1 \cdot) + t_2 \cos(k_2 \cdot) + v(t_1, t_2, r, p) \right)^2 \right\rangle
		\end{aligned}
	\end{equation}
	with the inner product term vanishing at $t_2=0$ since we assumed $k_2/k_1 \notin \NM_0$.
We can thus write
	\begin{equation}
		\label{e:Q2Factored}
		Q_2(t_1, t_2, r, p) = t_2 \, \Psi_2(t_1, t_2, r, p)
	\end{equation}
	with
	\begin{equation}
	\label{e:Psi2}
	\Psi_2(t_1,t_2,r,p) = \int_0^1 \frac{\partial Q_2}{\partial t_2}(t_1,z t_2, r, p) \;\mathrm{d}z
	\end{equation}
so that
	\begin{equation}
		\label{e:Psi2At00}
		\Psi_2(0, 0, r, p) = c_0 \left[ l((\kappa_0+p)k_2) - l(\kappa_0 k_2)\right] + r\, l((\kappa_0+p)k_2).
	\end{equation}
	
	Hence, condition \eqref{e:proj2D3} is equivalent solving the system
	\begin{equation*}
		\begin{cases}
			t_1 \Psi_1(t_1, t_2, r, p) = 0 \\
			t_2 \Psi_2(t_1, t_2, r, p) = 0
		\end{cases}
	\end{equation*}
	for $p$ and $r$ in a neighborhood of $(t_1,t_2,r,p)=(0,0,0,0)$.
	There are clearly four cases: $t_1 = t_2 = 0$ represents the trivial solutions.
	When $\Psi_1 = 0$ and $t_2=0$ we can apply Theorem \ref{thm:1Dlocal} concerning one-dimensional bifurcations along with
	the remark following it to obtain the solutions in $\mathcal{S}^{k_1}$.
	Similarly, when $t_1=0$ and $\Psi_2=0$ we instead retrieve the solutions in $\mathcal{S}^{k_2}$.
	To obtain the mixed-period solutions we apply the implicit function theorem to solve $\Psi_1 = \Psi_2 = 0$ near the origin.
	Indeed, note that $\Psi_1(0,0,0,0)=\Psi_2(0,0,0,0)=0$ and that the Jacobian of the map
	\[
	(r,p)\mapsto(\Psi_1(0,0,r,p),\Psi_2(0,0,r,p))
	\]
	at $(r,p)=(0,0)$ is given by
	\begin{align}
		\label{e:2DDet}
		\det&\left.
		\begin{bmatrix}
			D_r \Psi_1(0,0,r,p) & D_p \Psi_1(0,0,r,p) \\
			D_r \Psi_2(0,0,r,p) & D_p \Psi_2(0,0,r,p)
		\end{bmatrix}\right|_{(r,p)=(0,0)} \nonumber \\[5pt]
		& = c_0 \, l_{\kappa_0}( k_1) \left[ l'_{\kappa_0}( k_2)\, k_2 - l'_{\kappa_0}( k_1)\, k_1 \right],
	\end{align}
	which is always different from 0 since $l_T$ has only one positive stationary point,  $l_{\kappa_0}( k_1) \neq 0$,
	and that the terms  $l'_{\kappa_0}( k_1)$ and $ l'_{\kappa_0}( k_2)$ necessarily have opposite signs.
	Applying the Implicit Function Theorem gives	
	the solutions in $\mathcal{S}^{mixed}$.  Note in each of the above four cases, we find $r = r(t_1, t_2)$ and $p = p(t_1, t_2)$ with $p$ and $r$
	both vanishing to at least second order at $(t_1,t_2)=(0,0)$, as claimed.

\

\noindent
\emph{The resonant case}.	Assume now that $k_2/k_1 \in \NM_0$.  In this case, we are not guaranteed that $Q_2(t_1,0,r,p)=0$ for all $|t_1|\ll 1$ due to a possible
resonance in the inner product term in \eqref{e:Q2}.  Nevertheless, we do know that $Q_2(0,0,r,p)=0$.
Using polar coordinates to introduce the function
\[
\widetilde{Q}_2 (\varrho, \vartheta, r, p)=Q_2(\varrho\cos(\vartheta),\varrho\sin(\vartheta),r,p),
\]
defined for $0\leq \varrho\ll 1$ and $|(\vartheta,r,p)|\ll 1$, we find from \eqref{e:Q2} that
	\begin{align}
		\label{e:Q2Polar}
		\widetilde{Q}_2 (\varrho, \vartheta, r, p)&= \varrho \sin(\vartheta)  c_0 \big(l((\kappa_0+p)k_2) - l(\kappa_0 k_2)\big)  \\
			& \quad + \varrho \sin(\vartheta) r \, l((\kappa_0+p)k_2)  \nonumber \\
			& \quad - l((\kappa_0+p)k_2) \frac{1}{\pi}\int_{-\pi}^\pi \cos(k_2 x) \Big[ \varrho \cos(\vartheta) \cos(k_1 x) \nonumber \\
			& \quad + \varrho \sin(\vartheta) \cos(k_2 x) + v(\varrho \cos(\vartheta), \varrho \sin(\vartheta), r, p) \Big] ^2 \;\mathrm{d} x\nonumber.
	\end{align}
Since $\widetilde Q_2(0,\vartheta,r,p)=0$, we may as before write
	\begin{equation}
		\label{e:Q2FactoredPolar}
		\widetilde Q_2(\varrho, \vartheta, r, p) = \varrho \, \widetilde\Psi_2(\varrho, \vartheta, r, p)
	\end{equation}
	with
	\begin{equation}
		\label{e:Psi2Polar}
		\widetilde\Psi_2(\varrho,\vartheta,r,p) = \int_0^1 \frac{\partial \widetilde Q_2}{\partial \varrho}(z \varrho,\vartheta, r, p) \;\mathrm{d}z
	\end{equation}
	so that
	\begin{equation}
		\label{e:Psi2PolarAt0}
		\begin{aligned}
			\widetilde\Psi_2(0,\vartheta,r,p) &= \sin(\vartheta)\, c_0 \left[ l((\kappa_0+p)k_2) - l(\kappa_0 k_2)\right]  \\
									& \quad + r \sin(\vartheta)\, l((\kappa_0+p)k_2).
		\end{aligned}
	\end{equation}
	For $Q_1$, instead, all the previous calculations remain true and hence, similarly defining the function
\begin{equation}\label{e:Psi1Polar}
\widetilde{\Psi}_1 (\varrho, \vartheta, r, p):=\Psi_1(\varrho\cos(\vartheta),\varrho\sin(\vartheta)),
\end{equation}
	it follows in this resonant case that \eqref{e:proj2D3} is equivalent to solving the system
	\begin{equation*}
		\begin{cases}
			\varrho \cos(\vartheta) \widetilde\Psi_1(\varrho, \vartheta, r, p) = 0 \\
			\varrho \, \widetilde\Psi_2(\varrho, \vartheta, r, p) = 0.
		\end{cases}
	\end{equation*}
	for $r$ and $p$ in a neighborhood of $(\varrho,\vartheta,r,p)=(0,0,0,0)$.
	The case $\varrho=0$ clearly corresponds to trivial solutions, while the case $\cos(\vartheta)=0, \widetilde\Psi_2 = 0$ corresponds to solutions in $\mathcal{S}^{k_2}$ via the
	application of Theorem \ref{thm:1Dlocal}.
	For the case that $\widetilde\Psi_1=0, \widetilde\Psi_2=0$ we again apply the implicit function theorem near the origin.  
	Indeed, note that both $\widetilde\Psi_1$ and $\widetilde\Psi_2$ both vanish at the origin and that the Jacobian of the map
	\[
	(r,p)\mapsto(\widetilde\Psi_1(0,0,r,p),\widetilde\Psi_2(0,0,r,p))
	\]
	at $(r,p)=(0,0)$ is given by 
	\begin{align}
		\label{e:2DDetPolar}
		\det&\left.
		\begin{bmatrix}
			D_r \widetilde\Psi_1(0,\vartheta,r,p) & D_p \widetilde\Psi_1(0,\vartheta,r,p) \\
			D_r \widetilde\Psi_2(0,\vartheta,r,p) & D_p \widetilde\Psi_2(0,\vartheta,r,p)
		\end{bmatrix}\right|_{(r,p)=(0,0)} \nonumber \\[5pt]
		& = \sin(\vartheta) \, c_0 \, l(\kappa_0 k_1) \left[ l'(\kappa_0 k_2)\, k_2 - l'(\kappa_0 k_1)\, k_1 \right],
	\end{align}
	which, by the same considerations we applied to \eqref{e:2DDet}, is non-zero so long as $\sin(\vartheta) \neq 0$
	Therefore, for any fixed $\delta>0$, restricting to $\delta < |\vartheta| < \pi-\delta$ gives the solutions in $\mathcal{S}_\delta^{mixed}$, as desired
\end{proof}

\section{Global bifurcation diagram}\label{sec:discussion}
In this section we give some additional properties of solutions of \eqref{e:profile1}, that is, of continuous and finitely periodic solutions. Our goal is to communicate the global bifurcation picture, as gathered from both analytic and numerical evidence, as well as to relate this to some comparable studies. We first present and prove the additional analytic results, after which we discuss the bifurcation diagram of the periodic capillary-gravity Whitham with the help of Figures~\ref{fig:strong} and~\ref{fig:weak}.

\begin{proposition}\label{prop:smooth}
Any \(L^\infty(\RM)\)-solution of the steady capillary-gravity Whitham equation \eqref{e:profile1} is smooth.
\end{proposition}
\begin{proof}
This is immediate from writing the equation in the form \eqref{eq:whitham}. For any \(T > 0\), the operator \(L_T\) is a smoothing 
Fourier multiplier operator of order \(-\frac{1}{2}\). This applies in particular to the scale of Zygmund spaces 
\(\zygmund^s(\RM)\), \(s \geq 0\), see Lemma~\ref{lem:compactOperator}. As \(L^\infty(\RM)\) is an algebra embedded in \(\zygmund^{0}(\RM)\) \cite[Section 13.8]{Taylor2011pde}, and the spaces \(\zygmund^s(\RM)\) are Banach algebras for \(s > 0\), the result follows by bootstrapping.
\end{proof}

\begin{proposition}\label{prop:no solutions}
\( \)\\[-10pt]
\begin{itemize}
\item[(i)] 
There are no periodic solutions of \eqref{e:profile1} in the region
\[
\max u < \min\{0, c-1\}.\\[2pt]
\]
\item[(ii)]
Except for the bifurcation points when \(c = \frac{1}{l_T(k)} > 0\) there are no small periodic solutions in a vicinity of any point along the curve of trivial solutions \((u,c) = (0,c)\), \(c \in \RM\). Similarly, there are no periodic solutions that are small perturbations of the constant solutions \((u,c) = (c-1,c)\), \(c \in \RM\), except for the bifurcation points that appear along this line for \(c < 2\).\\[-6pt]
\item[(iii)]
The solution \(u=0\) is the only periodic solution for \(c=1\).\\[-6pt]
\item[(iv)] For $T\geq \frac{4}{\pi^2}$, all periodic solutions satisfy 
\[
\textstyle \max  u \leq \frac{c^2}{4}, 
\]
with equality if and only if \(u\) is a constant solution and either \(c = 0\) or \(c=2\).
\end{itemize}
\end{proposition}

\begin{remark}
The qualifier 'periodic' is here used only to guarantee that solutions, which we have defined to be continuous, are integrable over their period. 
\end{remark}

\begin{proof}
As all steady solutions are smooth, and the symbol of \(L_T\) satisfies \(l_T(0) = 1\), one may as in \cite{Ehrnstrom2016owc} integrate over any finite period to obtain 
\begin{equation}\label{eq:integral equality}
(c-1) \int_{-\pi}^{\pi} u\, \diff x =  \int_{-\pi}^{\pi} u^2\, \diff x.
\end{equation}
(The same argument works for other periods as well.) This is a contradiction for 
\(u < c - 1 < 0\). The analogous result for \(c > 1\) then follows by the Galilean invariance \(c \mapsto 2 - c\), \(u \mapsto u + 1 - c\).

For the second statement, consider first \(c < 1\). 
As the symbol \(l_T\) is positive, and the operator \(L_T\) is a linear isomorphism \(\zygmund^s(\SM) \to \zygmund^{s+\frac{1}{2}}(\SM)\) unless \(c l_T(k) = 1\) (cf. \eqref{eq:nontrivial kernel}), the implicit function theorem implies that there are no small solutions in a vicinity except for the bifurcation points found in Theorems~\ref{thm:1Dlocal} 
and~\ref{thm:2DLocal} when \(c < 1\).  In particular, there are no such solutions for \(c < 1\) in the case of strong surface tension \(T \geq \frac{1}{3}\), and 
none for \(c < 0\) in the case of weak surface tension \(0 < T < \frac{1}{3}\). By Galilean invariance, the corresponding result applies to the line \(u = c - 1\) for \(c \geq 1\). 

The proposition (iii) is immediate from \eqref{eq:integral equality}.  

For (iv), note that 
\[
u(x) = L (cu - u^2) = \frac{c^2}{4} - L \left(\frac{c}{2} -u \right)^2  \leq \frac{c^2}{4},
\]
when \(T \geq \frac{4}{\pi^2}\), as the integral kernel of \(L\) is then everywhere positive. This proves that \(\max u \leq \frac{c^2}{4}\), with equality if and only if \((u,c) = (1,2)\) or \((u,c) = (0,0)\), as these are the only constant solutions along the line \(\max u = \frac{c}{2}\). 
\end{proof}

\begin{proposition}\label{prop:leaving_c/2}
If the surface tension satisfies \(T \geq \frac{4}{\pi^2}\), then the bifurcation curve found in Theorem~\ref{thm:1Dglobal} for \(k=1\) can be constructed such that it contains a subsequence of solutions that are all single-crested (bell-shaped) in each minimal period and that either:\\[-12pt]
\begin{itemize}
\item[(i)] is bounded in wavespeed but with \(\min u\) unbounded; or\\[-10pt] 
\item[(ii)] eventually leaves every set \(\{\max u \leq \lambda c\}\) for \(\lambda < \frac{1}{2}\).
\end{itemize}
\end{proposition}

\begin{proof}
For even and periodic solutions \(u\) one may as in \cite{Ehrnstrom2016owc,EJC18} use \eqref{e:profile1} to write
\begin{equation}\label{eq:nodal}
u'(x) = 2\int_0^\pi\left(K_p(x-y)-K_p(x+y)\right)\left(\frac{c}{2}-u(y)\right)u'(y) \diff y.
\end{equation}
When \(K_p\) is completely monotone, and \(u\) is decreasing on $(0,\pi)$ with \(u \leq \frac{c}{2}\), this implies that \(u\) is strictly decreasing on the same interval (unless \(u\) is a constant), and a standard argument \cite[Lemma 5.5]{EJC18} yields that looping as in alternative (ii) is ruled out.  

Let us therefore, for a contradiction, assume that the bifurcation curve remains within the set \(\{\max u < \frac{c}{2}\}\).  Recalling that Theorem~\ref{thm:complete monotonicity} and
\cite[Proposition 3.2]{Ehrnstrom2016owc} together imply that $K_p$
is completely monotone on $(0,\pi)$ when \(T \geq \frac{4}{\pi^2}\), it follows that alternative (i) in Theorem~\ref{thm:1Dglobal} has to hold. As solutions are smooth, this is equivalent to a sequence of solutions \((u_n, c_n) = (u(t_n),c(t_n))\) satisfying \(|u_n|_\infty + |c_n| \to\infty\) as \(n \to \infty\).

Assume first that \(\{c_n\}_n\) is bounded. Then \(\{u_n\}_n\) is unbounded in \(L^\infty(\RM)\), and therefore \(\min u_n \to - \infty\) as \(n \to \infty\) is the only possibility, by Proposition~\ref{prop:no solutions}~(iv).

If, on the other hand, \(\{c_n\}_n\) is unbounded, pick a subsequence with \(\lim_{n \to \infty} |c_n| = \infty\). Note that \(c_n\) cannot pass \(c=1\), as Proposition~\ref{prop:no solutions}~(iii) shows that it would have to pass via \((u,c) = (0,1)\), but near that point there are only small constant solutions (see Remark~\ref{rem:transcritical} and Theorem~\ref{thm:2DLocal}). Hence, the solution curve would first have to connect to either the curve \(u = c-1\) or \(u = 0\). But, as described in Proposition~\ref{prop:no solutions}~(ii), the first of these has no bifurcation points for strong surface tension and \(c > 1\), and connection back to the bifurcation points of the second is excluded by the argument used in \cite[Lemma 5.5]{EJC18} (no looping). Hence, \(\lim_{n \to \infty} c_n = \infty\).

We now show that this is impossible when \(\max u_n \leq c_n/2\). Recall that we are following a branch of the curve for which \(u\) is even, and strictly increasing on the half-period \((-\pi,0)\), in view of the positivity of the integrand in \eqref{eq:nodal}. If there exists \(\delta > 0\) such that \(\frac{c_n}{2} - \max u_n \geq \delta c_n\), pick \(x_n \in (0,\pi)\) such that
\[
-u_n'(x_n) = \min_{y \in [\delta,\pi-\delta]} (-u_n'(y)).
\]

\begin{align*}
-u_n'(x_n) &= 2\int_0^\pi\left(K_p(x_n-y)-K_p(x_n+y)\right)\left(\frac{c_n}{2}-u_n(y)\right) (-u'(y)) \diff y\\
&\geq 2 \delta c_n \int_\delta^{\pi-\delta} \left(K_p(x_n-y)-K_p(x_n+y)\right) (-u_n'(y)) \diff y\\
&\geq -2 \delta c_n u_n'(x_n)   \int_\delta^{\pi-\delta} \left(K_p(x_n-y)-K_p(x_n+y)\right) \diff y.
\end{align*}
On the interval of consideration, \(K_p(x_n-y)-K_p(x_n+y)\) is bounded from below by a positive constant (it is zero only for \(y = k\pi\), \(k \in \ZM\)). Although it has a singularity at \(x_n = y\), it tends to \(\infty\) there, so we may estimate it from below, uniformly in \(x_n\), by
\[
\min \left\{\left(K_p(x_n-y)-K_p(x_n+y)\right) \colon (x,y) \in [\delta, \pi-\delta] \times [\delta, \pi-\delta]\right\}  \gtrsim 1.
\]
Consequently, 
\[
-u_n'(x_n) \gtrsim -c_n u_n'(x_n),
\]
which is not possible, as \(c_n \to \infty\) and \(-u_n'(x_n) > 0\) for all \(n\).
\end{proof}

\subsection{Discussion and summary of results}
Analytically, we have determined almost completely\footnote{We lack a proof of non-existence of the \(k_2\)-modal waves in the resonant case of Theorem~\ref{thm:2DLocal}, but these waves do not seem to exist numerically.} the solution set in near the lines of constant solutions \(u = 0\) and \(u = c -1\). The result depends crucially on the strength of surface tension \(T\), and, apart from the easily seen change in the dispersion relation at \(T =  \frac{1}{3}\), we have seen in Section~\ref{sec:kernel} that there is a more subtle change at \(T = \frac{4}{\pi^2}\), at which the integral kernel of the dispersive operator \(L\) loses its positivity and monotonicity; that has made it possible to prove some additional, but not complete, results for the case of (very) strong surface tension \(T \geq \frac{4}{\pi^2}\). To complete the picture where our analytical methods have so far proved insufficient, we have additionally run spectral bifurcation code similar to the one used in \cite{Remonato2017nbf} to get a more complete picture. We will present the main result of these calculations as well, but only in overview form.

\begin{figure}[t]
\begin{center}
\hspace*{-2em}\includegraphics[scale=0.5]{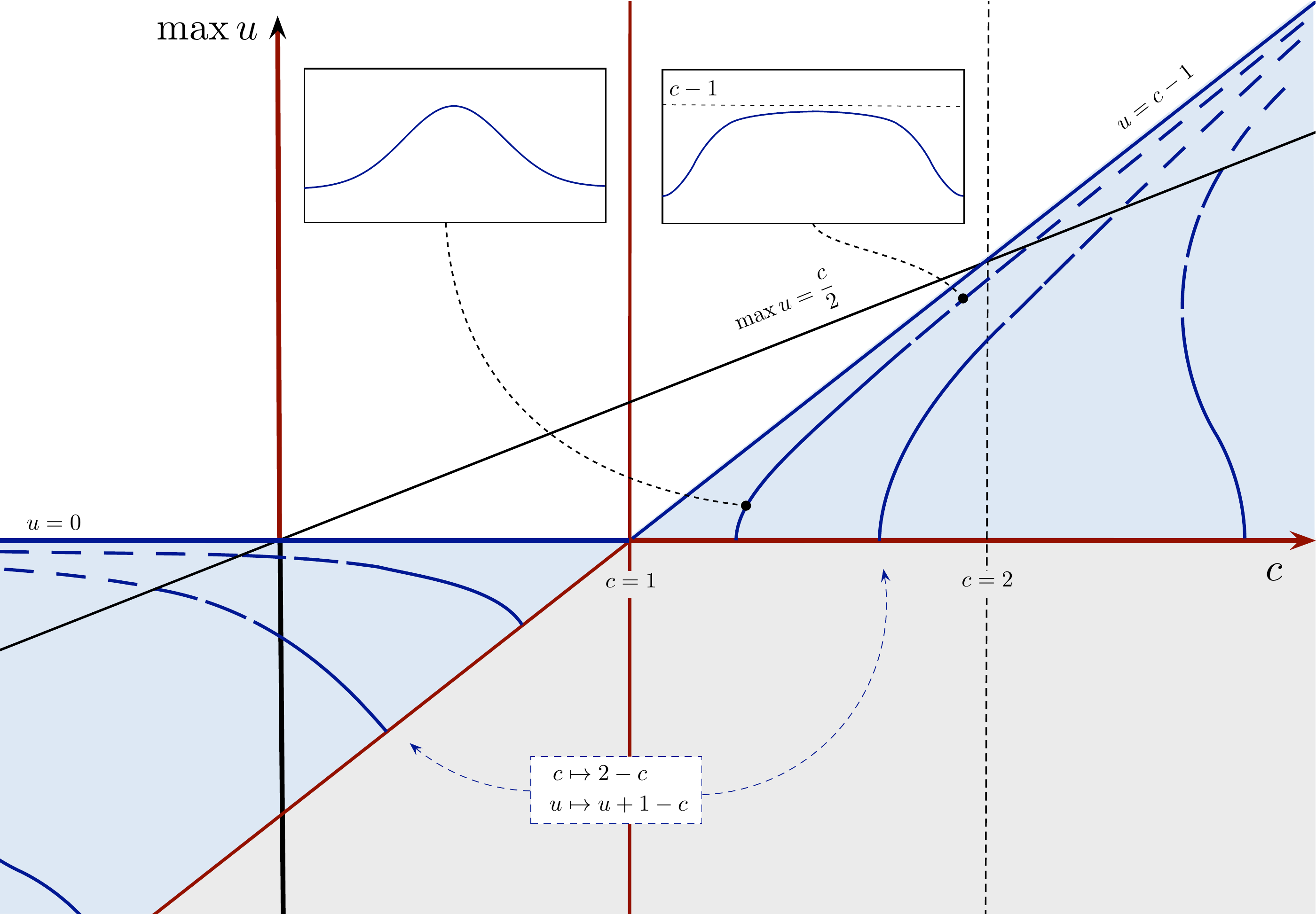}
\caption{A schematic drawing of the global bifurcation diagram in the case of strong surface tension $T>\frac{1}{3}$ (partly \(T = \frac{4}{\pi^2}\)). The diagram is discussed in detail in Section~\ref{subsec:strong}.}
\label{fig:strong}
\end{center}
\end{figure}

To start our discussion, focus first on one of the Figures~\ref{fig:strong} or~\ref{fig:weak}. Just as the regular Whitham equation, the capillary-gravity Whitham equation \eqref{e:profile1} admits two lines of constant solutions, namely \(u = 0\) and \(u = c-1\). These cross at \(c = 1\), the point of a transcritical bifurcation (see Remark~\ref{rem:transcritical}), and also a bifurcation point for solitary \cite{MR3485840} and generalised solitary \cite{arXiv:1807.11469} waves ; additionally, \(c=1\) is the symmetry line for the Galilean invariance
 \[
c \mapsto 2 - c, \qquad u \mapsto u + 1 - c, 
 \]
 that leaves \eqref{e:profile1} invariant, and is shared by the regular Whitham equation \cite{Ehrnstrom2016owc}. The two constants \(0\) and \(c-1\) correspond to the two natural depths that appear for steady flows in the water wave problem, see for example \cite{KK11}. In addition to these two lines, there is a third, mathematical, constant arising from the structure of \eqref{e:profile1} when completing the square, namely \(\frac{c}{2}\). While this constant is of physical and absolute importance in the regular Whitham equation --- being the height above surface of a highest wave --- and while it appears as a technical difficulty when trying to expand the result of Proposition~\ref{prop:leaving_c/2}, numerical evidence indicate that this construct is probably only artificial in the presence of capillarity. Still, we have indicated it in Figure~\ref{fig:strong} using the line \(\max u = \frac{c}{2}\) (but not in Figure~\ref{fig:weak}, as it did not prove any help in communicating our results). Additionally, in both Figures~3 and~4 the greyed-out area illustrates Proposition~\ref{prop:no solutions}, that there are no solutions in the region where 
 \[
 \max u < \min\{0, c-1\}.
 \]
 A final common feature of the strong and weak surface tension case is that solutions cannot pass \(c=1\), except via the transcritical bifurcation point \((u,c) = (0,1)\), where, locally, the only solutions are given by the constant functions \(u = 0\) and \(u = c-1\). This fact may be induced from Proposition~\ref{prop:no solutions}~(iii) and Remark~\ref{rem:transcritical}, and is indicated in the figures with a solid red line (no solutions pass). Note that both figures are for a fixed and finite period.

\subsection{The case of strong surface tension}\label{subsec:strong}
Now, let us focus on the strong surface tension case and especially the case \(T \geq \frac{4}{\pi^2}\), which is depicted in Figure~\ref{fig:strong}. As described in Theorem~\ref{thm:1Dlocal}, we have small waves of the approximate linear form \(\cos(k \cdot)\) bifurcating at 
\[
c_k = \frac{1}{l_T(k)} > 1.
\]
The bifurcation curves of these waves are indicated by solid blue lines, with a zoom-in on a small wave along the main bifurcation branch \(k=1\). The red line \(\{u = 0, 1 < c \neq c_k\}\) shows the result of Proposition~\ref{prop:no solutions}~(ii), that there are no other supercritical solutions in a \(\zygmund^s\)-vicinity of the line of vanishing solutions. By Galilean invariance, each of these curves (and non-existence results) has an exact counterpart for \(c < 1\) along the line \(u = c-1\), and we do not comment more on that in the case of strong surface tension.

The initial direction of the curves is calculated in Remark~\ref{rem:cdotdot}: analytically, sub-critical bifurcation is established for small enough values of \(k\), and super-critical bifurcation as \(k \to \infty\); numerically, this shift happens at exactly one value, and we have illustrated this with the last visible (third) curve bending leftwards from the bifurcation point, while the two first bend right-wards (the direction after the Galilean shift is opposite).

The result of the global bifurcation theory as carried out in Theorem~\ref{thm:1Dglobal} is that each curve, when considered in a space of \(2\pi/k\)-periodic functions, either is unbounded in \(\zygmund^s \times \RM\), or returns (loops) back to \((u,c) = (0,c_k)\) in a finite period of the bifurcation parameter. The standard tool for ruling out looping is by preserving the unimodal nodal pattern along the main bifurcation branch, an argument for which one relies on maximum principles/positivity of the underlying operators. As we prove in Theorem~\ref{thm:complete monotonicity} that this property is present when the surface tension coefficient satisfies \(T \geq \frac{4}{\pi^2}\) (and only then)\footnote{It is fully possible that the periodised kernel is positive even when the original kernel is not, depending on the period, but we have not investigated that here.}, the complete monotonicity of the kernel \(K\) established in Theorem~\ref{thm:complete monotonicity} for that case provides hope for stronger results. Note that, regardless of the exact value of \(T > 0\), it follows from Lemma~\ref{lem:compactOperator} that all solutions of \eqref{e:profile1} are smooth, so that alternative (i) in Theorem~\ref{thm:1Dglobal} is equivalent to a sequence of solutions satisying $|u|_\infty + |c| \to\infty$ along the bifurcation curve.
 
While we cannot rule out alternative (ii) in Theorem~\ref{thm:1Dlocal} completely, see Proposition~\ref{prop:leaving_c/2}, we can at least show that looping would require leaving every set of the form \(\max u < \lambda c\) for \(\lambda < \frac{1}{2}\) (that is the consequence of Proposition~\ref{prop:leaving_c/2}, as an unbounded continuous bifurcation curve cannot be finitely periodic). Although alternative (i) in Proposition~\ref{prop:leaving_c/2} is very unlikely, and never appears in our numerical calculations, we have been unable to rule it out (the reason for this might be that the balance between \(Mu\) and \(u^2\) is exactly at the critical threshold for Gagliardo--Nirenberg, so that control of a higher Sobolev norm of \(u\) in terms of a lower seems to require using precise properties of the integral kernel.) We have illustrated this with long-dashed lines in Figure~\ref{fig:strong}, showing the curves (probably) leaving the cone \(\max u \leq \frac{c}{2}\). 

After that point, our calculations are purely numerical, showing the solution curves asymptotically approaching the second curve of constant solutions \(u = c -1\). Indeed, if the quotient  
 \[
\frac{\max(u)}{c-1}
\]
is bounded and converges pointwise, it is immediate from \eqref{eq:integral equality} that the limit is either \(0\) or \(1\).The numerics indicate that this quotient increases along the bifurcation curve to cover all of the interval \((0,1)\), with wave profiles that are monotone on a minimal half-period even though, by far, we have passed \(u = \frac{c}{2}\). Such a result, we believe, would be new in the setting of capillary-gravity water waves, but it is so far out of reach for us when \(u\) crosses \(\frac{c}{2}\). Interestingly enough, the same pattern seems to persist even when the kernel is not everywhere positive and monotone, that is, for \(T < \frac{4}{\pi^2}\).

Finally, for surface tension \(T \geq \frac{4}{\pi^2}\), Proposition~\ref{prop:no solutions} shows that no solutions pass the line \(c=0\) with \(\max u \geq 0\), indicated by red in Figure~\ref{fig:strong}.

\subsection{The case of weak surface tension}\label{subsec:weak}
When the surface tension is weak, \(T < \frac{1}{3}\), several things are very different. First of all, the first single bifurcation points \(c_k\) might, depending on the period, appear in the interval \(0 < c < 1\), although for large enough values of the wavenumber \(k\) the waves will all be supercritical. Just as as in the case of strong surface tension, Proposition~\ref{prop:no solutions} guarantees that solutions do not cross the lines marked with red in Figure~\ref{fig:weak} (although these now do not include the positive vertical axis \(\max u > 0\)), and there are no solutions in the grey area. Similarly, there are no small, non-constant, solutions in a neighbourhood of any point along the constant solution axes \(u = 0\) and \(u = c-1\), except at the countable bifurcation points.

A peculiarity in the case of weak surface tension is the appearance of multimodal waves connecting different curves of \(k\)-modal waves. Analytically, we find a full disk of solutions by two-dimensional bifurcation in Theorem~\ref{thm:2DLocal}~(i), by varying the wavelength. Fixing the fundamental period, however, this yields a one-dimensional subset of this disk, where we continuously transform via only a curve between two main modes of waves. Numerically, this effect persists even for values slightly off the exact points of two-dimensional bifurcation: as the numerical investigation \cite{Remonato2017nbf} shows, the looping alternative (i) in the global one-dimensional Theorem~\ref{thm:1Dglobal}  happens in the form of one bifurcation curve of \(k\)-modal waves transforming into one of \(n\)-modal waves and thereby connecting back to the line of zero states. The same kind of connections have been found for the Euler equations, analytically for small waves \cite{JonesToland85}, and numerically for small and large waves \cite[Figures 4 and 5]{Aston1991lag} (see also \cite{MR768482,MR724024} for perturbation theory and numerical calculations showing the rippling and non-uniqueness of small waves). These branch-to-branch connections are illustrated in Figure~\ref{fig:weak} by a curve of small bimodal waves connecting two curves of unimodal waves bifurcating off the \(0\)-axis for \(c \in (0,1)\). (In numerical calculations for this manuscript, there have even been instances of curves of waves bridging, consecutively, three different unimodal bifurcation curves, that is, a nontrivial path that connects three separate bifurcation points, but that is not indicated in the graphics.) 

The curves of subcritical waves can be followed, again numerically, past zero wave speed, going left-ward without any indication to stop. In \(L^2(\SM)\), they seem to flatten out to \(0\), but not in \(L^\infty\). This feature reappears again and again in both numerics and our calculations: while \(L^\infty\)-bounds easily yield bounds on higher norms, and one has control of solutions in \(L^2\) with respect to the wave speed, it is extremely difficult to relate the \(L^\infty\)-norm of solutions to their \(L^2(\SM)\)-norm, even when the wave speed is bounded. Generally, all curves of solutions appear to asymptotically approach one of the curve of constant solutions (\(u = 0\) or \(u = c-1\)) in \(L^2(\SM)\), while an actual connection in a space of higher regularity is impossible for almost all wavespeeds because of the invertibility of the linear operator \(D_u F\) (note that it is not obvious how to make sense of the nonlinear mapping \(F\) in \(L^2(\SM)\)). 

Finally, in the case of supercritical bifurcation, we find only single-crested (bell-shaped) waves even though the surface tension is weak. When these waves are small it is a result of Theorem~\ref{thm:1Dlocal}. These curves may be continued globally (Theorem~\ref{thm:1Dglobal}), but the information about them is purely numerical. Just as in the case of strong surface tension, these supercritical waves show no ripples, and they asymptotically approach \(u = c -1\) in \(L^2(\SM)\), but not in \(L^\infty\). Any proof of preservation of the nodal properties in the case of supercritical bifurcation when the surface tension is weak is for the moment entirely out of our reach, even though it would be very interesting to obtain.

\begin{figure}[t]
\begin{center}
\hspace*{-2em}\includegraphics[scale=0.5]{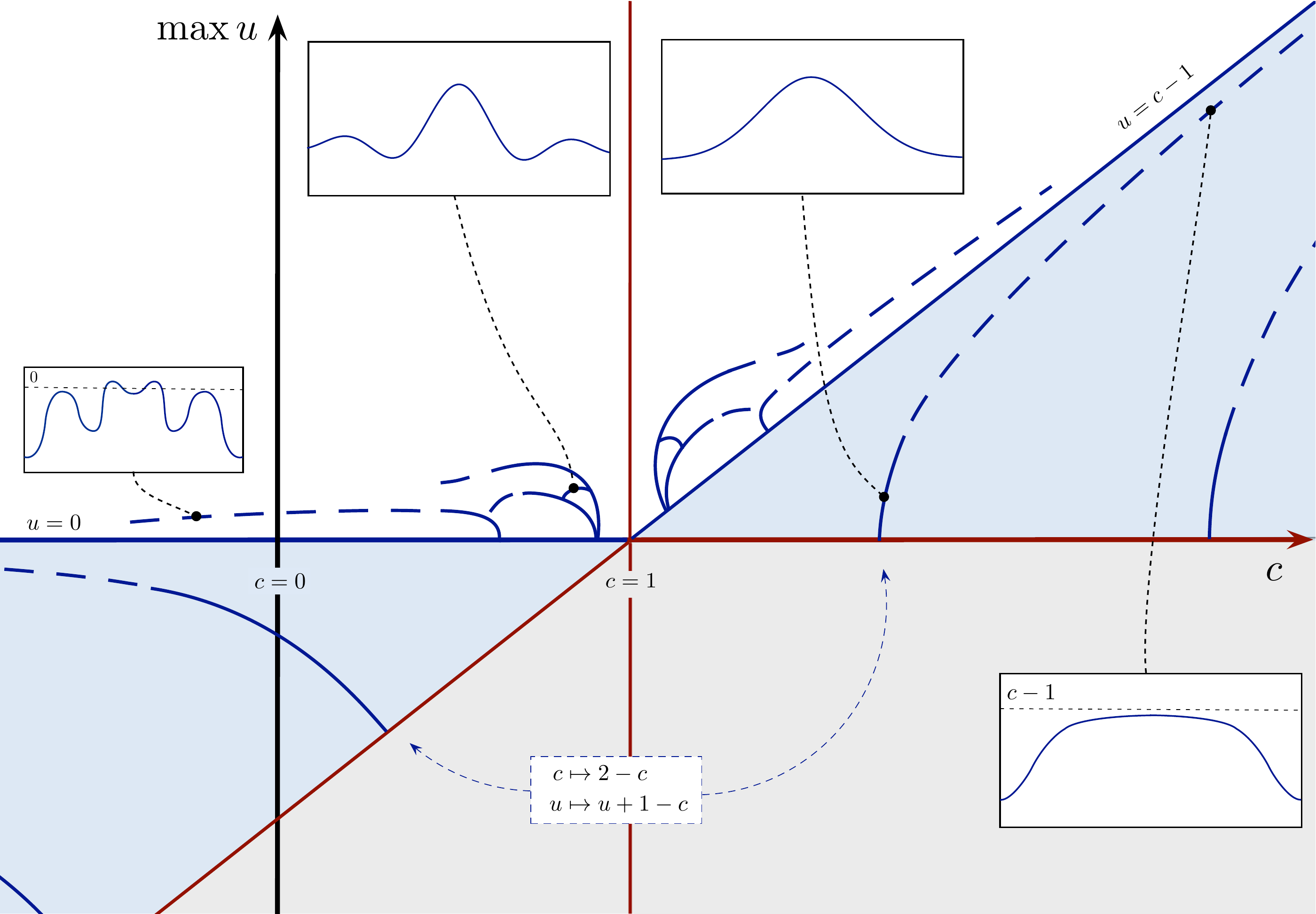}
\caption{A schematic drawing of the global bifurcation diagram in the case of weak surface tension $T<\frac{1}{3}$. The diagram is discussed in detail in Section~\ref{subsec:weak}.}
\label{fig:weak}
\end{center}
\end{figure}

\appendix

\section{Bifurcation formulas}\label{app:A}

This appendix contains higher order expansions of the quantities in Theroem \ref{thm:1Dlocal} and Theorem \ref{thm:2DLocal}.
We start with the first and second order terms in the expansion for the speed $c(t)$ in the one-dimensional bifurcation case, 
which is required by the proof of the global extension in Theorem \ref{eq:global condition}.  We then proceed to study the first 
order terms for the expansions of the functions $r$ and $p$ in the two-dimensional bifurcation case.  

\subsection{One-dimensional bifurcation case}
We begin by determining the derivatives $\dot{c}(0)$ and \(\ddot{c}(0)\) associated to the bifurcation curve constructed in Theorem~\ref{thm:1Dlocal}. This can be done either directly using the Lyapunov--Schmidt reduction carried out in the proof of~Theorem~\ref{thm:1Dlocal} or by the means of bifurcation formulas given for example in \cite{MR2004250}. The latter requires an identification between the bifurcation function $\phi(u,c) = \Pi F(u + \psi(u,c), c)$ used in \cite{MR2004250} and the functions \(v\) and \(r\) used in the proof of  Theorem~\ref{thm:1Dlocal}. This relation is given by \(v(t) = \psi(t \cos(k x), c(t))\). 

Here, start from the Lyapunov--Schmidt representation
\begin{equation}\label{e:LS}
\begin{aligned}
	0 &= F(t\cos(kx) + v(t) , c_0 + r(t))\\
		&= t\cos(kx) + v(t)\\ 
		&\quad+ L \left[ (t\cos(kx) + v(t))^2 -(c_0 + r(t)) (t\cos(kx) + v(t)) \right],
\end{aligned}
\end{equation}
where here it is understood that for each $t$ small the function $v(t)$ is a $2\pi/k$-periodic function of $x$.
Differentiating \eqref{e:LS} once with respect to $t$, evaluating at $t=0$ and using that $v(0) = \dot v(0) = r(0) = 0$ yields the equation
\[
\left(1-c_0L\right)\cos(kx)=0,
\]
which holds by our choice of $c_0$.  Similarly, differentiating \eqref{e:LS} twice with respect to $t$ and evaluating at $t=0$ yields
\begin{equation}\label{e:vdd}
\begin{aligned}
(1-c_0L)\ddot{v}(0)&=2\dot{r}(0)L\cos(kx)-2L\cos^2(kx)\\
&=2\dot{r}(0)l(k)\cos(kx)-\left(1+l(2k)\cos(2kx)\right).
\end{aligned}
\end{equation}
Since $\int_{-\pi}^\pi v(t)\cos(kx)dx=0$ for all $|t|\ll 1$, the above implies that $\dot{r}(0)=0$.  Returning to \eqref{e:vdd}, it now follows that
\begin{equation}\label{e:vdd_explicit}
\ddot{v}(0)=\frac{1}{c_0-1}+\frac{l(2k)\cos(2kx)}{c_0l(2k)-1}.
\end{equation}

Continuing, we observe that taking the third derivative of \eqref{e:LS} with respect to $t$ and evaluating at $t=0$ yields
\[
\left(1-c_0L\right)\dddot{v}(0)=3\ddot{r}(0)L\cos(kx)-6L\left(\ddot{v}(0)\cos(kx)\right).
\]
Using \eqref{e:vdd_explicit}, we compute that
\[
L\left(\ddot{v}(0)\cos(kx)\right)=\frac{l(k)\cos(kx)}{c_0-1}+\frac{l(2k)\left(l(k)\cos(kx)+l(3k)\cos(3kx)\right)}{2(c_0l(2k)-1)}.
\]
Using again that $\int_{-\pi}^\pi v(t)\cos(kx)\,\diff x=0$ for all $|t|\ll 1$, it follows that
\[
\ddot{r}(0)=\frac{3}{c_0-1}+\frac{l(2k)}{c_0l(2k)-1}=\frac{3c_0l(2k)-l(2k)-2}{(c_0-1)(c_0l(2k)-1)},
\]
which is the expression \eqref{eq:global condition} for $\ddot{c}(0)$ given in Theorem~\ref{thm:1Dglobal}.  Note that the above procedure could be continued
to obtain asymptotic expansions of $r(t)$ and $v(t)$ to arbitrarily high order in $t$.  We also note that the above result is consistent with the asymptotic formulas
in \cite{Hur2015mii}.

\subsection{Two-dimensional bifurcation case}\label{S:2d_der_calc}

We now consider the case of a two-dimensional bifurcation as considered in Section \ref{sec:twodim} above.  Recall that the solutions constructed 
in Theorem \ref{thm:2DLocal} can be written as
\begin{align*}
	u(t_1, t_2) &= t_1 \cos(k_1 x) + t_2 \cos(k_2 x) + v(t_1, t_2), \\
	c(t_1, t_2) &= c_0 + r(t_1, t_2), \\
	\kappa(t_1, t_2) &= \kappa_0 + p(t_1, t_2),
\end{align*}
with $v$ of order $\mathcal{O}(|(t_1, t_2)|^2)$ and $r$, $p$ of order $\mathcal{O}(|(t_1,t_2)|)$.  We now characterize the order
of vanishing of the functions $r$ and $p$ at the origin.

\begin{proposition}\label{L:2D_der_calc}
Let the functions $r$ and $p$ be as in Theorem \ref{thm:2DLocal}.  	If $k_2/k_1 \notin \NM_0$, then
	\begin{equation*}
		\nabla r(0,0) = 0, \qquad \nabla p(0,0) = 0
	\end{equation*}
	so that, in particular, $r$ and $p$ are of order $\mathcal{O}(|(t_1,t_2)|^2)$ near the origin.
	If instead $k_2/k_1 \in \NM_0$, then for any $\delta>0$ small we have that, in polar coordinates,
	\begin{equation*}
		r_\varrho\left(0,\vartheta\right) = 0, \qquad p_\varrho\left(0,\vartheta\right) = 0
	\end{equation*}
	if and only if either $k_2 \notin \{0, 2k_1 \}$ or $(k_2,\vartheta)=\left(2k_1,\frac{\pi}{2}\right)$.
\end{proposition}

\begin{proof}
We begin the non-resonant case, $k_2/k_1 \notin \NM_0$.  From the proof of Theorem \ref{thm:2DLocal}, we know for all $0<|(t_1,t_2)|\ll 1$ the functions
$r$ and $p$ satisfy 
\begin{equation*}
		\Psi_i (t_1, t_2, r(t_1, t_2), p(t_1, t_2)) = 0\quad{\rm for}~~ i=1,2,
\end{equation*}
where the $\Psi_i$ are defined in \eqref{e:Psi1} and \eqref{e:Psi2}.
Fixing $j\in\{1,2\}$ we find that differentiating the above with respect to $t_j$ and evaluating at $(t_1,t_2)=(0,0)$ gives the system of equations
\begin{equation}\label{e:rp_system}
\left(\begin{array}{cc}\Psi_{1,r}({\bf 0}) & \Psi_{1,p}({\bf 0})\\
									\Psi_{2,r}({\bf 0}) & \Psi_{2,p}({\bf 0})\end{array}\right)
\left(\begin{array}{c}r_{t_j}(0,0)\\p_{t_j}(0,0)\end{array}\right)
=
-\left(\begin{array}{c}\Psi_{1,t_j}({\bf 0})\\ \Psi_{2,t_j}({\bf 0})\end{array}\right),
\end{equation}
where here ${\bf 0}$ denotes the origin in $\RM^4$.  Since the above system matrix is invertible by \eqref{e:2DDet},
it remains to determine the values of $\Psi_{i,t_j}({\bf 0})$ for $i=1,2$.  This can be accomplished by recalling \eqref{e:Psi1} and
\eqref{e:Psi2} and noting that \eqref{e:Q1} implies that
\[
\frac{\partial^2 Q_i}{\partial t_j^2}({\bf 0}) = -\frac{2}{\pi}\,l(\kappa_0 k_i) \int_{-\pi}^\pi \cos^3(k_i x) \;\mathrm{d}x
\]
and
\[
\frac{\partial^2 Q_i}{\partial t_1 \partial t_2}({\bf 0}) =
\left\{\begin{aligned}
 -\frac{2}{\pi} l(\kappa_0 k_2) \int_{-\pi}^\pi \cos^2(k_1 x) \cos(k_2 x) \;\mathrm{d}x,\quad i=1,\\
  -\frac{2}{\pi} l(\kappa_0 k_1) \int_{-\pi}^\pi \cos^2(k_2 x) \cos(k_1 x) \;\mathrm{d}x,\quad i=2
  \end{aligned}\right.
\]
Consequently, since $k_2/k_1\notin\NM_0$ it follows that $\Psi_{i,t_j}(\vec{0})=0$ for $i=1,2$ and hence
\eqref{e:rp_system} implies that $r_{t_j}(0,0)=p_{t_j}(0,0)=0$ as claimed.  Since $j\in\{1,2\}$ was arbitrary, this proves
the proposition in the non-resonant case.

Now, consider the resonant case when $k_2/k_1\in\NM_0$ and fix $\delta>0$ small.  In this case, for each $\delta<|\vartheta|<\pi-\delta$ and $0<\varrho\ll 1$ the functions
$r(\varrho,\vartheta)$ and $p(\varrho,\vartheta)$ satisfy the system 
\[
\widetilde\Psi_i\left(\varrho,\vartheta,r(\varrho,\vartheta),p(\varrho,\vartheta)\right)=0\quad{\rm for}~~i=1,2,
\]
where here the $\widetilde\Psi_i$ are as in \eqref{e:Psi1Polar} and \eqref{e:Psi2Polar}.	Differentiating this system with respect to $\varrho$ at $\varrho=0$ gives the
system of equations
\begin{equation}\label{e:rp_system_polar}
\left(\begin{array}{cc}\widetilde\Psi_{1,r}(0,\vartheta,0,0) & \widetilde\Psi_{1,p}(0,\vartheta,0,0)\\
									\widetilde\Psi_{2,r}(0,\vartheta,0,0) & \widetilde\Psi_{2,p}(0,\vartheta,0,0)\end{array}\right)
\left(\begin{array}{c}r_\varrho(0,\vartheta)\\p_\varrho(0,\vartheta)\end{array}\right)
=
-\left(\begin{array}{c}\widetilde\Psi_{1,\varrho}(0,\vartheta,0,0)\\ \widetilde\Psi_{2,\varrho}(0,\vartheta,0,0)\end{array}\right).
\end{equation}
As in the non-resonant case, the above system matrix is invertible, this time thanks to \eqref{e:2DDetPolar}, and hence it remains to determine the values
of $\widetilde\Psi_{i,\varrho}(0,\vartheta,0,0)$ for $i=1,2$.   	Let us begin by determining the value in the case $i=1$.  From \eqref{e:Psi1Polar} and the preceding discussion,
we know we can write
\[
\widetilde\Psi_1(\varrho,\vartheta,0,0)=\int_0^1\frac{\partial \widetilde Q_1}{\partial \varrho}(z\varrho,\vartheta,0,0)\,\diff z
\]
where, using \eqref{e:Q1}, we have explicitly 
\begin{align*}
\widetilde Q_1(\varrho,\vartheta,0,0)&=Q_1(\varrho\cos(\vartheta),\varrho\sin(\vartheta),0,0)\\
&=-\frac{2\varrho^2 l(k_0k_1)\cos(\vartheta)\sin(\vartheta)}{\pi}\int_{-\pi}^\pi\cos^2(k_1 x)\cos(k_2x)\,\diff x.
\end{align*}
Clearly then,  $\widetilde{Q}_{2,\varrho\varrho}(0,\vartheta,0,0)$ is equal to zero if and only if either $\vartheta=\frac{\pi}{2}$ or $k_2\notin\{0,2k_1\}$.
Since
\[
\widetilde\Psi_{1,\varrho}(0,\vartheta,0,0)=\frac{1}{2}\frac{\partial^2\widetilde Q_1}{\partial\varrho^2}(0,\vartheta,0,0)
\]
by above, we have shown that  $\widetilde\Psi_{1,\varrho}(0,\vartheta,0,0)=0$ if and only if either of the conditions $\vartheta=\frac{\pi}{2}$ or $k_2\notin\{0,2k_1\}$ hold.

Similarly, we have
\[
\widetilde\Psi_{2,\varrho}(0,\vartheta,0,0)=\frac{1}{2}\frac{\partial^2\widetilde Q_2}{\partial\varrho^2}(0,\vartheta,0,0)
\]
where, using \eqref{e:Q2Polar}, we have
\begin{align*}
\widetilde Q_2(\varrho,\vartheta,0,0)&=-\frac{\varrho^2 l(k_0k_2)}{\pi}\times\\
&\qquad\int_{-\pi}^\pi\cos(k_2 z)\left[\cos^2(\vartheta)\cos^2(k_1 x)+\sin^2(\vartheta)\cos^2(k_2 x)\right]\diff x.
\end{align*}
Clearly, $\widetilde{Q}_{2,\varrho\varrho}(0,\vartheta,0,0)$ vanishes whenever $k_2\notin\{0,2k_1\}$.  When $k_2=0$, 
$\widetilde{Q}_{2,\varrho\varrho}(0,\vartheta,0,0)$ does not vanish for any $\vartheta$, and when $k_2=2k_1$ it only vanishes when $\vartheta=\frac{\pi}{2}$.  Consequently,
$\widetilde\Psi_{2,\varrho}(0,\vartheta,0,0)$ vanishes only when either  $k_2\notin\{0,2k_1\}$ or  $\left(k_2,\vartheta\right)=(2k_1,\frac{\pi}{2})$.
Together with the results concerning $\widetilde\Psi_{1,\varrho}$, this completes the proof.
\end{proof}

\begin{remark}
	The special case $k_2 = 2k_1$ has been found also in the Euler equations (with gravity and vorticity) by the authors of \cite{Aasen2018tgw}.
	The special case $k_2 = 0$ is instead due to the transcritical double bifurcation allowed by the capillary-gravity  Whitham equation.
\end{remark}

\begin{remark}
An explicit example where $r_\varrho(0,\vartheta) \neq 0$ can be seen in \cite[Figure 6]{Remonato2017nbf}, where the branch of nontrivial solutions has a non-vertical tangent at the bifurcation point in the speed-height plane.
\end{remark}

\bibliographystyle{siam}
\bibliography{EJMR_capillary}

\begin{thebibliography}{10}

\bibitem{Aasen2018tgw}
{\sc A.~Aasen and K.~Varholm}, {\em {Traveling Gravity Water Waves with
  Critical Layers}}, J. Math. Fluid Mech, 20 (2018), pp.~161--187.

\bibitem{arXiv:1808.08057}
{\sc M.~N. Arnesen}, {\em A non-local approach to waves of maximal height for
  the {Degasperis-Procesi} equation}.
\newblock arXiv:1808.08057.

\bibitem{MR3485840}
{\sc M.~N. Arnesen}, {\em Existence of solitary-wave solutions to nonlocal
  equations}, Discrete Contin. Dyn. Syst., 36 (2016), pp.~3483--3510.

\bibitem{Aston1991lag}
{\sc P.~Aston}, {\em {Local and global aspects of the (1,n) more interaction
  for capillary-gravity waves}}, Physica D, 52 (1991), pp.~415--428.

\bibitem{Bhatia2007pdf}
{\sc R.~Bhatia}, {\em {Positive definite matrices}}, Princeton Series in
  Applied Mathematics, Princeton University Press, Princeton, NJ, 2007.

\bibitem{BuffoniToland}
{\sc B.~Buffoni and J.~Toland}, {\em {Analytic Theory of Global Bifurcation -
  An Introduction}}, Princeton Series in Applied Mathematics, Princeton
  University Press, 2003.

\bibitem{Dancer1973btf}
{\sc E.~Dancer}, {\em {Bifurcation theory for analytic operators}}, Proc. Lond.
  Math. Soc, XXVI (1973), pp.~359--384.

\bibitem{Dancer1973gso}
\leavevmode\vrule height 2pt depth -1.6pt width 23pt, {\em {Global structure of
  the solutions set of non-linear real-analytic eigenvalue problems}}, Proc.
  Lond. Math. Soc, XXVI (1973), pp.~747--765.

\bibitem{eew10}
{\sc M.~Ehrnstr{\"o}m, J.~Escher, and E.~Wahl{\'e}n}, {\em {Steady water waves
  with multiple critical layers}}, SIAM J. Math. Anal., 43 (2011),
  pp.~1436--1456.

\bibitem{07022009}
{\sc M.~Ehrnstr{\"o}m, H.~Holden, and X.~Raynaud}, {\em Symmetric waves are
  traveling waves}, Int. Math. Res. Not.,  (2009), pp.~4578--4596.

\bibitem{EJC18}
{\sc M.~Ehrnstr\"{o}m, M.~A. Johnson, and K.~M. Claassen}, {\em Existence of a
  highest wave in a fully dispersive two-way shallow water model}, Archive for
  Rational Mechanics and Analysis,  (2018).
\newblock https://doi.org/10.1007/s00205-018-1306-5.

\bibitem{Ehrnstrom2009twf}
{\sc M.~Ehrnstr\"{o}m and H.~Kalisch}, {\em {Traveling waves for the Whitham
  Equation}}, Differentialand Integral Equations, 22 (2009), pp.~1193--1210.

\bibitem{Ehrnstrom2013gbf}
\leavevmode\vrule height 2pt depth -1.6pt width 23pt, {\em {Global Bifurcation
  for the Whitham Equation}}, Mathematical Modelling of Natural Phenomena, 8
  (2013), pp.~13--30.

\bibitem{Ehrnstrom2015tsw}
{\sc M.~Ehrnstr\"{o}m and E.~Wahl\`en}, {\em {Trimodal Steady Water Waves}},
  Arch. Rational Mech. Anal., 216 (2015), pp.~449--471.

\bibitem{Ehrnstrom2016owc}
\leavevmode\vrule height 2pt depth -1.6pt width 23pt, {\em {On Whitham's
  conjecture of a highest cusped wave for a nonlocal dispersive equation}},
  arXiv,  (2016).
\newblock arXiv:1602.05384.

\bibitem{MR3243741}
{\sc L.~Grafakos}, {\em {Modern {F}ourier analysis}}, vol.~250 of Graduate
  Texts in Mathematics, Springer, New York, third~ed., 2014.

\bibitem{MR724024}
{\sc J.~K. Hunter and J.-M.~a. Vanden-Broeck}, {\em Solitary and periodic
  gravity---capillary waves of finite amplitude}, J. Fluid Mech., 134 (1983),
  pp.~205--219.

\bibitem{Hur2015mii}
{\sc V.~M. Hur and M.~A. Johnson}, {\em {Modulational instability in the
  {W}hitham equation with surface tension and vorticity}}, Nonlinear Anal., 129
  (2015), pp.~104--118.

\bibitem{arXiv:1807.11469}
{\sc M.~A. {Johnson} and J.~D. {Wright}}, {\em Generalized solitary waves in
  the gravity-capillary {W}hitham equation}.
\newblock arXiv:1807.11469.

\bibitem{Katznelson2004ait}
{\sc Y.~Katznelson}, {\em {An introduction to harmonic analysis}}, Cambridge
  Mathematical Library, Cambridge Univerisity Press, third~ed., 2004.

\bibitem{MR2004250}
{\sc H.~Kielh{\"o}fer}, {\em {Bifurcation theory}}, vol.~156 of Applied
  Mathematical Sciences, Springer-Verlag, New York, 2004.

\bibitem{KK11}
{\sc V.~Kozlov and N.~Kuznetsov}, {\em Steady free-surface vortical flows
  parallel to the horizontal bottom}, Q. J. Mechanics Appl. Math., 64 (2011),
  pp.~371--399.

\bibitem{Kozlov2017nms}
{\sc V.~Kozlov and E.~Lokharu}, {\em {N-Modal Steady Water Waves with
  Vorticity}}, Journal of Mathematical Fluid Mechanics,  (2017), pp.~1--15.

\bibitem{MR3060183}
{\sc D.~Lannes}, {\em The water waves problem}, vol.~188 of Mathematical
  Surveys and Monographs, American Mathematical Society, Providence, RI, 2013.
\newblock Mathematical analysis and asymptotics.

\bibitem{Martin2013eow}
{\sc C.~Martin and B.-V. Matioc}, {\em {Existence of Wilton Ripples for Water
  Waves with Constant Vorticity and Capillary Effects}}, SIAM Journal on
  Applied Mathematics, 73 (2013), pp.~1582--1595.

\bibitem{Reeder1981owr}
{\sc J.~Reeder and M.~Shinbrot}, {\em {On Wilton ripples, II: Rigorous
  results}}, 77 (1981), pp.~321--347.

\bibitem{Remonato2017nbf}
{\sc F.~{Remonato} and H.~{Kalisch}}, {\em {Numerical Bifurcation for the
  Capillary Whitham Equation}}, Physica D, 343 (2017), pp.~51--62.

\bibitem{Schilling2012bf}
{\sc R.~L. Schilling, R.~Song, and Z.~Vondra\^cek}, {\em {Bernstain
  Functions}}, vol.~37 of de Gruyter Studies in Mathematics, Walter de Gruyter
  \& Co., Berlin, second~ed., 2012.

\bibitem{Taylor2011pde}
{\sc M.~E. Taylor}, {\em {Partial Differential Equations III. Nonlinear
  Equations}}, Applied Mathematical Sciences, Springer, second~ed., 2011.

\bibitem{JonesToland85}
{\sc J.~F. Toland and M.~C.~W. Jones}, {\em The bifurcation and secondary
  bifurcation of capillary-gravity waves}, Proc. Roy. Soc. London Ser. A, 399
  (1985), pp.~391--417.

\bibitem{Trichtchenko2016tio}
{\sc O.~Trichtchenko, B.~Deconinck, and J.~Wilkening}, {\em {The instability of
  Wilton ripples}}, Wave Motion, 66 (2016), pp.~147--155.

\bibitem{MR768482}
{\sc J.-M. Vanden-Broeck}, {\em Nonlinear gravity-capillary standing waves in
  water of arbitrary uniform depth}, J. Fluid Mech., 139 (1984), pp.~97--104.

\bibitem{Whitham1967vma}
{\sc G.~Whitham}, {\em {Variational methods and applications to water waves}},
  Proc. R. Soc. Lond. A, 299 (1967), pp.~6--25.

\end{thebibliography}

\end{document}